\documentclass[11pt,a4paper]{article}
\usepackage[top=60pt,bottom=60pt,left=82pt,right=82pt]{geometry}

\usepackage{lineno,hyperref}
\newif\ifcomm
\commtrue
\usepackage{comment}
\usepackage{hyperref}
\usepackage{epsfig}
\usepackage{graphicx}
\usepackage{amsmath, graphicx, latexsym, amssymb, amsthm, amsfonts}
\usepackage{amssymb}
\usepackage{amsmath}
\usepackage{amsfonts} %Gal: i added this to use mathbb
\usepackage[tight,footnotesize]{subfigure}
\usepackage{bbm}
\usepackage{url}
\usepackage{algorithm}
\usepackage[noend]{algpseudocode}
\usepackage{times}
\usepackage{todonotes}
\usepackage{xspace}
\usepackage{balance}
\ifcomm
    \newcounter{commentNumberI}
     \newcommand{\Isaac}[1]{\addtocounter{commentNumberI}{1}{{({\color{blue} {(\arabic{commentNumberI}.)} Isaac: #1})}}} %just color
      \newcommand{\Rami}[1]{\addtocounter{commentNumberI}{1}{{({\color{orange} {(\arabic{commentNumberI}.)} Rami: #1})}}} %just color
       \newcommand{\Ariel}[1]{\addtocounter{commentNumberI}{1}{{({\color{brown} {(\arabic{commentNumberI}.)} Ariel: #1})}}} %just color
       \newcommand{\Gal}[1]{\addtocounter{commentNumberI}{1}{{({\color{red} {(\arabic{commentNumberI}.)} Gal: #1})}}} %just color
      \newcommand{\Shay}[1]{\addtocounter{commentNumberI}{1}{{({\color{green} {(\arabic{commentNumberI}.)} Shay: #1})}}} %just color
    \newcommand{\A}[1]{\textbf{[#1]}}       %Comments -> in bold

    \usepackage[normalem]{ulem}
    \newcommand{\del}[1]{\textcolor{red}{X: \textbf{\sout{#1}}}} %be careful, this can leave double spaces in the document when in comments

\else
    \newcommand{\A}[1]{}
    \newcommand{\Isaac}[1]{}
    \newcommand{\Gal}[1]{}
    \newcommand{\Ariel}[1]{}
	\newcommand{\Shay}[1]{}
    \newcommand{\Rami}[1]{}
    
    \newcommand{\del}[1]{} %be careful, this can leave double spaces in the document when in comments
\fi

\newcommand{\remove}[1]{}

    \def\compactify{\itemsep=0pt \topsep=0pt \partopsep=0pt \parsep=0pt}
      \let\latexusecounter=\usecounter

\newtheorem{theorem}{Theorem}[section]
\newtheorem{lemma}[theorem]{Lemma}

\newtheorem{proposition}[theorem]{Proposition}

\newtheorem{remark}[theorem]{Remark}
\newtheorem{key observation}[theorem]{Key observation}

\newcommand{\be}{\begin{equation}}
\newcommand{\ee}{\end{equation}}

\newcommand\MyIncludeGraphics[2][]{% needs packages: \usepackage{graphicx}\usepackage{todonotes}
    \IfFileExists{#2}{%
        \includegraphics[#1]{#2}%
    }{%
        \missingfigure[figwidth=7.0cm]{Missing #2}%
    }%
}%

\newcommand{\T}[1]{\smallskip\noindent\textbf{#1}} %paragraph title
\newcommand{\para}[1]{\left( #1 \right)}        %Shortcuts in equations for parentheses

\newcommand{\eps}{\varepsilon}

\newcommand{\E}{\mathbb{E}}  %\newcommand{\E}{\text{E}} , \newcommand{\E}{\text{\textbf{E}}}       % Expected value
\newcommand{\p}[1]{\mathbb{P} \para{#1}}  % Probability
\newcommand{\pp}[2]{\mathbb{P}_{#2} \para{#1}}  % Probability
\newcommand{\newVar}[2]{\newcommand{#1}{\ensuremath{#2}\xspace}}
  \newVar{\server}{S}
  \newVar{\client}{C}
  \newVar{\rclient}{R_c}
  \newVar{\rserver}{R_s}

\providecommand{\ie}{\emph{i.e.,} }

\providecommand{\iid}{\emph{i.i.d.}\xspace}

\providecommand{\keywords}[1]
{
  \small	
  \textbf{\textit{Keywords---}} #1
}

\title{On the Persistent-Idle Load Distribution Policy Under Batch Arrivals and Random Service Capacity}
\author{Rami~Atar$^{(2)}$, Isaac~Keslassy$^{(2)}$, Gal~Mendelson$^{(1)}$, Ariel~Orda$^{(2)}$, Shay~Vargaftik$^{(3)}$ \smallskip\smallskip\\ \small 1. Graduate School of Business, Stanford University \\ \small 2. Viterbi Faculty of Electrical Engineering, Technion\\ \small 3. VMware Research\\}
\begin{document}
\date{}
\maketitle

%=========================================================================
%  Abstract
%=========================================================================

\begin{abstract}
The \textit{Persistent-Idle} (PI) load distribution policy was recently introduced as an appealing alternative to current low-communication load balancing techniques. In PI, servers only update the dispatcher when they become idle, and the dispatcher always sends jobs to the last server that reported being idle. PI is unique in that it does not seek to push the server queue lengths towards equalization greedily. Rather, it aggressively pulls the servers away from starvation. As a result, PI's analysis requires different tools than other load balancing approaches. So far, PI was proven to achieve the stability region for Bernoulli arrivals and deterministic and constant service capacities.

Our main contribution is proving that PI achieves the stability region in a model with batch arrivals and random service capacities. Proving this result requires developing tighter bounds on quantities of interest and proving the fulfillment of a generalized version of the state-dependent drift criteria previously considered. We also present PI-Split, which allows batches to be split among idle servers and prove its stability. Finally, we conduct simulations comparing PI and PI-Split to other load balancing policies. Simulation results indicate that both achieve excellent performance for the model under consideration for a wide range of parameters.

\end{abstract}

%=========================================================================
%  key words
%=========================================================================

\keywords{Parallel Server Model; Load Balancing; Persistent-Idle; Stability; Lyapunov Drift}

%\end{frontmatter}

%\linenumbers

%=========================================================================
%  Introduction
%=========================================================================

\section{Introduction}
\label{sec_introduction}
The Persistent Idle (PI) load distribution policy was recently introduced in \cite{PI1} as an appealing reduced-state load balancing scheme. It is designed for the parallel server model, in which jobs arrive at a dispatcher and are routed to one of $n$, possibly heterogeneous, servers. In PI, a job is routed to an idle server if there is one or to the last server a job was sent to, otherwise. Thus the dispatcher only uses idleness information to make its routing decisions making the communication cost at most a single message per job.

Despite this low communication overhead, PI was shown to achieve the stability region when servers are heterogeneous; namely, with possibly different processing rates.
In their considered model the following assumption were made: at most one job may arrive at a time slot; each job's workload depends on the server it is assigned to and equals the number of time slots required to process it on that server; each server completes one unit of work at each time slot.
Thus at most one departure can occur in each server at each time slot.

In this paper, we study PI in a considerably different, well-established model with batch arrivals
and random service capacities (e.g. \cite{ying2017power}).
An additional variant of the policy is considered, called PI-Split, in which the dispatcher may split jobs belonging to the same batch to several servers.
Our main results are that PI and PI-Split achieve the stability region in this model.

Proving that PI achieves the stability region in this model presents a significant challenge. In particular, the proof in \cite{PI1} relies on sampling the Markov chain describing the state of the system at the random times at which new servers become idle, and proving a state-dependent drift criteria is fulfilled by a carefully chosen Lyapunov function. This proof heavily relies on the fact that, given the current state, the next sampling time is deterministic. 
This is no longer the case in the model we consider.

\begin{comment}
For example, considering the model in \cite{PI1}, suppose that there are three servers and the workloads in the system are given by $(0,3,5)$. PI will send all incoming jobs to server 1 until servers 2 or 3 become idle. Since in this model the servers complete 1 unit of work at each time slot, it will take exactly 3 time slots for server 2 (and 5 time slots for server 3, respectively) to become idle. Thus, the next sampling time is known given the state at the current sample.

 Now, consider our current model, and suppose that the number of jobs in the system is given by $(0,3,5)$. PI still sends all incoming jobs to server 1, but due to the random service capacities, it is not clear which server (2 or 3) will become idle first and how many time slots this will take. Even given the current state, the answer depends on the random service capacities: it could be that server 3 can process 10 jobs in the next time slot while server 2 can process only one, in which case server 3 becomes idle first in one time slot.  Alternatively, if both servers can process only one job at each of the next 3 time slots, server 2 becomes idle first in 3 time slots.
\end{comment}

Our main contribution is proving the fulfillment of a generalized version of a state-dependent drift criteria (we discuss this at length in Section \ref{subsubsec:Overview}) and proving it is sufficient for stability. 
Our proof includes developing tighter bounds on quantities of interest and a method to relate the current queue length state to the time interval's duration until the next sampling time.

Finally, we present simulation results comparing PI and PI-Split to other load-distribution policies that were considered in \cite{PI1}. 
Simulation results indicate that PI has a comparable performance to Join-the-Shortest-Queue (JSQ) up to high loads, and favorable performance compared to other well-known reduced-state load balancing schemes, i.e., Join-the-Idle-Queue (JIQ) \cite{lu2011join}, Power-of-choice (JSQ(2)) \cite{mitzenmacher2001power} and Power-of-Memory (JSQ(1,1)) \cite{shah2002use}.

The remainder of this paper is organized as follows. Section \ref{sec:system models} contains the system model and the definition of the PI and PI-Split policies. Section \ref{sec:main results} contains the stability results for PI and PI-Split. Section \ref{sec:simulations} contains the simulation results, comparing PI and PI-Split to other load-distribution policies.

%%%%%%%%%%%%%%%%%%%%%%%%%%%%%%%%%%%%%%%%%%%%%%%%%%%%
%%%%%%%%%%%%%%%%%%%%%%%%%%%%%%%%%%%%%%%%%%%%%%%%%%%%
\section{Model}\label{sec:system models}
\label{sec:liq}

\subsection{Persistent Idle}\label{subsec: PI system model}

We consider a parallel-server system that evolves in discrete time $t \in \{0,1,2\ldots\}$ with a single dispatcher and $n$ heterogeneous work-conserving servers. Each of the $n$ servers has an unbounded FIFO queue holding pending jobs. 

\T{Order of events.} The order of events during each time slot is as follows. (1) New jobs may arrive at the dispatcher; (2) the dispatcher immediately decides on a server to which the jobs will be assigned to and routes them accordingly; (3) servers provide service. In particular, servers can process jobs that have just arrived; (4)  
communication between the dispatcher and the servers may occur. As a result, the dispatcher may update its state (e.g. which servers have no pending jobs), on which it may base the assignment decision at the next time slot.

\T{Batch arrivals.}  The job arrival process is given by a sequence of \iid random variables (RVs) $\{a(t)\}_{t=1}^{\infty}$. Namely, at each time slot $t>0$, $a(t)$ jobs arrive as a \textit{batch} at the dispatcher. We assume $\p{a(1)=0}>0$ to avoid dealing with a periodic Markov chain. 
We denote the mean and standard deviation of $a(1)$ by $\lambda$ and $\sigma_a$, respectively.

\T{Random service capacities.} 
The service process of server $i\in\{1,\ldots,n\}:=[n]$ is given by a sequence of \iid RVs $\{s_i(t)\}_{t=1}^{\infty}$, representing the {potential} number of jobs that a server can complete in a time slot. We assume that these RVs are positive and bounded, \ie there exists $s_{max}\in \mathbbm{N}$ such that 
\begin{equation}\label{eq:s assumption}
    1 \leq s_i(t)\leq s_{max}, \quad \mbox{ for all } i\in [n] \mbox{ and } t>0.
\end{equation}
We denote the mean and standard deviation of $s_i(1)$ by $\mu_i$ and $\sigma_{s_i}$, respectively. Assumption \eqref{eq:s assumption} implies
\begin{equation}\label{eq:mu assumption}
    1 \leq \mu_i\leq s_{max}, \quad \mbox{ for all } i\in [n].
\end{equation}

\T{Queue lengths and idleness.} Let $Q_i(t)$ denote the number of pending jobs in queue $i$ at the \textit{end} of time slot $t$, \textit{after} arrival and departure. Let $Q(t)=(Q_1(t),\ldots,Q_n(t))$, and $Q=(Q(t))_{t=0}^{\infty}$. We refer to $Q(0)$ as the initial condition. Server $i$ is \textit{idle} at time slot $t$ if $Q_i(t)=0$. A server \textit{becomes} idle at time slot $t>0$ if it had jobs to process during time slot $t$ and processed all of them. For example, if $Q_i(t-1)=0$, jobs arrived to server $i$ at the beginning of time slot $t$ and it served all of them, i.e. $Q_i(t)=0$, then it became idle at time slot $t$. But, if $Q_i(t-1)=0$ and $Q_i(t)=0$ and no jobs arrived at server $i$, it did not become idle at time slot $t$. 

\T{Tokens.} A single token is associated with each server. Thus, the total number of tokens in the system is $n$. Each token can be held by either the corresponding server or the dispatcher. When a server becomes idle, it sends its token to the dispatcher. Similarly, when the dispatcher chooses an idle server as a destination for a job, it implicitly sends the token back to this server along with the job. At time slot $t=0$ the dispatcher holds the tokens corresponding to the servers that are initially idle.

%For example, consider the following two cases. \\

\noindent \textit{Example 1}. If $Q_i(t-1)>0$, then server $i$ must hold its token at the beginning of time slot $t$.
If, in addition, $Q_i(t)=0$, server $i$ becomes idle at time slot $t$ and sends the token at the end of time slot $t$.

\noindent \textit{Example 2}. If $Q_i(t-1)=0$, the dispatcher must hold server $i$'s token at the beginning of time slot $t$. If, in addition, jobs arrive to server $i$ at the beginning of time slot $t$ then the dispatcher sends the token to server.
If server $i$ now completes processing all jobs, it sends the token
at the end of time slot $t$. Thus, in this case, the token travels from the dispatcher
and back at the same time slot.

\T{PI routing policy.}  
The routing is encoded in a process that takes values in $[n]$, denoted by $LI=(LI(t))_{t=0}^{\infty}$. If a job arrives at time $t>0$, it joins queue $LI(t)$.
We refer to the corresponding server as the \textit{Last-Idle} server. Roughly speaking, the Last-Idle server is the one that was idle most recently.
If there are multiple idle servers, the Last-Idle server is chosen by some tie-breaking rule. 
For concreteness we assume that the tie-breaking rule is that the Last-Idle server is chosen uniformly at random among the tied servers. The exact mathematical definition of the process $LI$ appears below.
Thus, the policy persistently sends the arriving jobs to the Last-Idle server until other servers become idle.

%%% I think we don't need this and it messes up LI(0) %%%
\begin{comment}
To define this procedure, denote by $\sigma_i(t)$ as the last time slot at which server $i$ is idle before time slot $t$. Formally, 
\begin{equation}\label{eq:def of sigma}
  \sigma_i(t)=\max\{s: 0 \leq s <t, \ Q_i(s)=0\},  
\end{equation}
where if the set is empty, then $\sigma_i(t)=-1$.
Then, the Last-Idle server at time $t$ is given by \begin{equation}\label{eq:def of LI}
    LI(t)=argmax_i\{\sigma_i(t)\},
\end{equation}
where in case of a tie, the Last-Idle server is chosen by some tie-breaking rule. 
For ease of exposition, we assume that the tie-breaking rule is that the Last-Idle server is chosen uniformly at random. Define the Last-Idle process as $LI=(LI(t))_{t=0}^{\infty}$. 

By \eqref{eq:def of sigma} and \eqref{eq:def of LI}, $\sigma_i(0)=-1$ for all $i$. Thus $LI(0)$ is chosen uniformly at random from $\{1,\ldots,n\}$. 
With regard to the value of $LI(1)$, assuming some initial condition $Q(0)$, then $\sigma_i(1)=0$ if $Q_i(0)=0$ and $\sigma_i(1)=-1$ if $Q_i(0)>0$. By the definition of $LI(t)$, the value of $LI(1)$ is then chosen uniformly at random from the set of idle servers at $t=0$. If there are none, then it is chosen uniformly at random from the set of all servers.
\end{comment}
%%%%%%%%%%%%%%%%%%%%%%%%%%%%%%%%%%%%%%%%%5

\T{Markov chain formulation.}
The precise mathematical formulation is given in terms of a Markov chain, that
we now define.
%The state of the chain is given by the pair
The Markov chain, denoted $\{X(t)\}=\{(Q(t),LI(t))\}$, has state space
\[
\hat{\mathcal{A}}=\mathbb{N}^n \times [n].
\]
A member $\alpha$ of the state space $\hat{\mathcal{A}}$ is of the form $$\alpha=(\alpha_1,\ldots,\alpha_n,l),$$ where $\alpha_1,\ldots,\alpha_n$ are queue lengths, and $l$ is the identity of the Last-Idle server.

Define the set of idle servers at the end of time slot $t$ by
\begin{equation}\label{eq:idle servers}
    I(t)=\{i:Q_i(t)=0\}.
\end{equation}
The queue length and Last-Idle processes satisfy the following recursion for $t>0$, 
\begin{align}\label{eq:maxflow}
\begin{cases}
Q_i(t)=[Q_i(t-1)+a(t)\mathbbm{1}_{\{LI(t)=i\}}-s_i(t)]^+ \\ 
LI(t)=\begin{cases} LI(t-1),& \text{if }\vert I(t-1)\vert=0 \,\\ 
 \xi_t,&\text{otherwise,}
\end{cases}
\end{cases}
\end{align}
where for each $t>0$, $\xi_t$ is a RV
%that is independent of $(a(t),s_1(t),\ldots,s_n(t))$, and
whose conditional distribution given
\[
(X(0),X(1),\ldots,X(t-1),a(t),s_1(t),\ldots,s_n(t))
\]
is uniform on the set $I(t-1)$ on the event $\{\vert I(t-1)\vert>0\}$,
and $\xi_t=1$ on the event $\{\vert I(t-1)\vert=0\}$
(note that the value of $\xi_t$ on the latter event is immaterial; it is specified only for concreteness). It is clear from the definition that $X$ forms a Markov chain.

We define the set ${\cal{A}} \subset \hat{\cal{A}}$ as the collection of all states $\alpha\in\hat{\cal{A}}$ such that $\alpha$ can be reached
with a positive probability from any of the states of the form $(0,\ldots,0,i)$, $i\in [n]$, corresponding to an empty system.
The initial condition $X(0)=(Q(0),LI(0))$ is assumed to lie in $\cal{A}$ with probability 1.
Under this assumption, we may and will assume without loss of generality
that the state space of $X$ is ${\cal A}$.

By \eqref{eq:maxflow}, the definition of $\cal{A}$ and the fact that the process can reach $(0,\ldots,0,i)$
from $(0,\ldots,0,j)$ for any $i,j$ (thanks to the assumptions that $\p{a(1)=0}>0$ and the uniformly-at-random tie-breaking rule),
it follows that $X$ is an irreducible, aperiodic, homogeneous Markov chain on ${\cal A}$.

\T{Departures.} Let $d_i(t)$ denote the
%actual
number of jobs completed by server $i$ during time slot $t$.
Clearly, $d_i(t)\leq s_i(t)$. Moreover, $d_i(t)$ is given by 
$$d_i(t)=\min\{s_i(t),Q_i(t-1)+a(t)\mathbbm{1}_{\{LI(t)=i\}}\},$$
and the first equation in recursion \eqref{eq:maxflow} can also be written as 
\begin{equation}\label{flow}
Q_i(t)=Q_i(t-1)+a(t)\mathbbm{1}_{\{LI(t)=i\}}-d_i(t).
\end{equation}

%%%%%%%%%%%%%%%%%%%%%%%%%%%%%%%%%%%%%%%%%%%%%%%%%%%%
%%%%%%%%%%%%%%%%%%%%%%%%%%%%%%%%%%%%%%%%%%%%%%%%%%%%

\subsection{Persistent Idle with batch splitting}\label{subsec: PI split model}
The PI policy sends all of the jobs belonging to a batch to a single server. The batch-splitting PI policy (referred to as \textit{PI-Split}) is a variation
in which whenever there are at least two idle servers, the dispatcher is allowed to split an arriving batch into \textit{sub-batches} and send them to multiple idle servers. If there are none, the dispatcher sends the batch to the server corresponding to the value of $LI(t)$. 

Under PI-Split, the tokens are handled exactly as under PI: when a server becomes idle, it sends its token to the dispatcher and when a dispatcher sends jobs to a server (either a batch or a sub-batch) it implicitly returns the token. Thus, under PI-Split, whenever a batch is split into sub-batches, multiple tokens are returned to the corresponding servers. 

The updating rule for $LI$ remains the same as under PI: If there are idle servers, $LI$ is chosen uniformly at random. If there are none, its value remains unchanged. We add a superscript $S$ for the batch splitting case such that the queue lengths, idleness and Last-Idle processes are given by $Q^S$, $I^S$ and $LI^S$ respectively. The state of the system is now given by the process $$X^S=:(Q^S,LI^S).$$ 

The only difference between PI and PI-split is how the queue lengths are updated whenever $\vert I^S(t-1) \vert>0$. To keep the Markov chain formulation simple for this case, we assume that the dispatcher determines how to split the batch between the idle servers based only on the current queue length state and the current batch size. 

Formally, let $\{B(t)\}_{t=1}^\infty$ be an i.i.d process taking values in $\mathbbm{R}$. Let $v[i]$ denote the $i$th member of a vector $v$ in $\mathbbm{N}^n$. Define a family of deterministic functions 
\begin{equation}\label{eq:f def}
  f:\mathbbm{N}^n\times \mathbbm{N}\times \mathbbm{R}\rightarrow \mathbbm{N}^n  
\end{equation}
satisfying the following conditions:
\begin{align}\label{eq:f cond}
&\sum_{i=1}^n f(\alpha_1,\ldots,\alpha_n,a,b)[i]=a,\cr
&f(\alpha_1,\ldots,\alpha_n,a,b)[i]=0 \quad \text{ if } \min_{1\leq j \leq n}{\alpha_j}=0 \text{ and }\alpha_i>0.
\end{align}
The inputs to $f$ are the queue lengths $Q^S$, the batch size $a(t)$ and the value of the i.i.d process $B(t)$. The $i$th member of the output vector equals the number of jobs the dispatcher should send to server $i$. The first condition in \eqref{eq:f cond} states that the total number of assigned jobs is exactly the batch size. The second condition states that if there are idle servers ($\min_{1\leq j \leq n}\alpha_j=0$), then any server that is not idle should not receive any jobs. The behaviour of $f$ in the case where there are no idle servers is immaterial. 

Following the definitions above, we can now state the updating rule for $X^S$:
%%%%%%%%%%%%%%%%%%%%%%%%%%%%%%%%%%
\begin{align}\label{eq:one recurrsion split}
\begin{cases}
Q^S_i(t)=[Q^S_i(t-1)+a_i(t)-s_i(t)]^+\\
LI^S(t)=\begin{cases} LI^S(t-1),& \text{if }\vert I^S(t-1)\vert=0 \,\\ 
 \xi^S_t,&\text{otherwise,}
\end{cases}\\
a_i(t)=\begin{cases} a(t),& \text{if } \vert I^S(t-1)\vert=0 \text{ and } i=LI^S(t)\,\\ 
f(X^S_i(t-1),a(t),B(t))[i],& \text{if } \vert I^S(t-1)\vert>0 \text{ and } i \in I^S(t-1)\,\\ 0, & \text{otherwise}
\end{cases}
\end{cases}
\end{align}

where $\xi^S_t$ is a RV distributed uniformly on the idle servers and is defined analogously to $\xi_t$ in \eqref{eq:maxflow} and the function $f$ encodes the batch splitting rule and belongs to the family of functions defined by \eqref{eq:f def} and \eqref{eq:f cond}. Thus, $X^S$ is an irreducible, aperiodic, time homogeneous Markov chain. 

%%%%%%%%%%%%%%%%%%%%%%%%%%%%%%%%%%%%%%%%%%%%%%%%%%%%%%%%%%%%%%%%%%%%
\section{Main results} \label{sec:main results}

\subsection{Communication overhead} \label{subsec:Communication overhead}

We begin with a simple result regarding the communication overhead of PI and PI-Split. Both only rely on tokens passed from the servers to the dispatcher.

\begin{proposition}\label{prop:pi}
The communication overhead of PI (resp. PI-Split) consists of at most a single message per batch (resp. sub-batch).
\end{proposition}
\begin{proof}
A necessary condition for a transmission of a token is the completion of a job. Since PI does not split batches, a completion of a job which results in a server becoming idle must be a completion of a batch. Similarly, for PI-Split, a message can be sent only when sub-batches are completed and not all completions result in a message.
\end{proof}

%%%%%%%%%%%%%%%%%%%%%%%%%%%%%%%%%%%%%%%%%%%%
\subsection{Stability results} \label{subsec:Stability results}

We introduce the following notation.
For  $\alpha \in \mathcal{A}$, we write $\mathbb{E}_\alpha[\ {\cdot} \ ]$ and 
$\pp{\ {\cdot} \ }{\alpha}$ for the conditional expectation $\mathbb{E}[\ {\cdot} \ |X(0){=}\alpha]$ and probability 
$\p{\ \cdot \ |X(0){=}\alpha}$, 
respectively.
\subsubsection{Overview}\label{subsubsec:Overview}

A useful approach to proving stability for irreducible Markov processes relies on the construction of a Lyapunov function $\cal{L}$, mapping the state space to $[0,\infty)$, that satisfies a drift condition within a single time step, of the form
\begin{equation}\label{eq:single time step drift}
\E_\alpha[{\cal{L}}(X(1))]-{\cal{L}}(\alpha)<-\eps<0,
\end{equation}
such that the drift is negative and bounded away from zero at all states $\alpha$ outside of some finite set $A\subset \cal{A}$, and a finite expectation at all states $\alpha\in A$, namely $\E_\alpha[{\cal{L}}(X(1))]<\infty$.
The existence of such a function ensures that provided the chain starts in the finite set $A$, the expected return time of the chain to $A$ is finite. This, in turn, implies several stability properties,
such as positive recurrence and convergence to a unique invariant measure (Lemma I.3.10 of \cite{asmussen2008applied}).

However, this approach does not appear to be directly applicable for the model under consideration.
As discussed at length in \cite{PI1}, the main difficulty arises when considering states where one queue length is large while the others are close to zero. It is possible that the large queue length belongs to the Last-Idle server, which receives all of the work. 

The proof in \cite{PI1} relies on sampling the underlying Markov chain at (random) times at which the identity of the Last-Idle server may change, namely, the times at which a token is available at the dispatcher and it is not the last used token. If the service capacities are deterministic (as in  \cite{PI1}),
%given the current system state,
the next such time is a function of the current state alone. Then, the calculation of the drift is performed with respect to time intervals that depend (deterministically) on the states the process traverses (as in \cite{malyshev1979ergodicity}, \cite{meyn1994state}). 

At these sampling times the queue length vector is always at the boundary of the positive orthant. The drift condition then only needs to be verified at these special states, rather than at the whole space.
Denote the set of these states by ${\cal{A}}_b$ (where $b$ is mnemonic for boundary). Because the Lyapunov function is observed at sampling times, the negative drift condition must be stated in terms of the duration of the intervals. That is, in place of \eqref{eq:single time step drift}, it is shown in \cite{PI1} that
\begin{equation}\label{eq:state dep drift}
\E_\alpha[{\cal{L}}(X(\tau))]-{\cal{L}}(\alpha)<-\eps\tau,
\end{equation}
for all but finitely many $\alpha$ in the set ${\cal{A}}_b$, where $\eps>0$ and $\tau$ denotes the next sampling time.

The random service capacities we consider in this paper further complicate the analysis. In this case, given the current state $X(0)=\alpha$, the next sampling time $\tau$ is random. Instead of \eqref{eq:state dep drift}, our method is based on showing that
\begin{equation}\label{eq:state dep drift random}
\E_\alpha[{\cal{L}}(X(\tau))]-{\cal{L}}(\alpha)<-\eps\E_\alpha[\tau],
\end{equation}
for all but finitely many $\alpha$ in the set ${\cal{A}}_b$.

We analyze the stability of PI and PI-Split in Sections \ref{subsec: stability of PI} and \ref{subsec: stability of PI split} respectively. 

 %%%%%%%%%%%%%%%%%%%%%%%%%%%%%%%%%%%%%%%%%%%
%%%%%%%%%%%%%%%%%%%%%%%%%%%%%%%%%%%%%%%%%%%%%%%%%%%%%

\subsection{Stability of PI} \label{subsec: stability of PI}

\begin{theorem}\label{main result PI}
Assume $\lambda < \sum_{i=1}^n\mu_i$. Then: \\
(i) $X$ is positive recurrent. Consequently, \\
(ii) $X$ has a unique stationary distribution, denoted by $\pi_X$, and \\
(iii) For any initial state $\alpha\in\mathcal{A}$ and any $B \subset \mathcal{A}$, $\mathbb{P}_\alpha(X(t)\in B) \rightarrow \pi_X(B)$ as $t \rightarrow \infty$.
\end{theorem}

\noindent \textit{Proof.} 
For $A\subset \cal{A}$ define the hitting time
\begin{equation}\label{eq:tau A def}
\tau^A=\inf{\{t\geq 1; \, X(t) \in A \}}.
\end{equation}
We prove that there exists a nonempty finite set $A\subset {\cal{A}}$ such that for all $\alpha_0 \in A$,
\begin{equation}\label{eq:tau}
\mathbbm{E}_{\alpha_0}[\tau^A]<\infty.
\end{equation}
By I.3.10 of \cite{asmussen2008applied}, the irreducibility of $X$ and \eqref{eq:tau} imply positive recurrence
and hence part (i) of Theorem \ref{main result PI}.
Part (ii) then follows by I.3.6 of \cite{asmussen2008applied}.
Finally, since the chain is aperiodic, part (iii) follows by the ergodic theorem for Markov chains, Theorem I.4.2 of \cite{asmussen2008applied}. It thus remains to show \eqref{eq:tau}. 
To this end, we define a sampled chain $Y$
of $X$, sampled at the times when the value of $LI$ may change.

%%%%%%%%%%%%%%%%%%%%%%%%%%%%%%%%
%%%%%%%%%%%%%%%%%%%%%%%%%%%%%%%%%%%%%%%%%
To define the random sampling times, we first define the following filtrations. The first is the natural filtration of $X$, namely
\begin{equation}\label{eq:filtration_natural}
{\cal{F}}_t=\sigma(X(0), \ldots ,X(t))
\end{equation}
and the second also includes the RVs $\xi_1,\ldots,\xi_{t+1}$ which are used to break ties between the idle servers, namely
\begin{equation}\label{eq:filtration}
{\cal{F}}^{\xi}_t=\sigma(X(0), \ldots ,X(t),\xi_1,\ldots,\xi_{t+1}).
\end{equation}
The filtration ${\cal{F}}^{\xi}_t$ in \eqref{eq:filtration} is useful when knowing the identity of the newly used token simplifies the analysis. 

By \eqref{eq:maxflow}, the value of $LI$ can only change at time $t$ if there is an idle server which is not the Last-Idle server, corresponding to the event  
\begin{equation}\label{eq:omega}
    \Omega_t:=\{\exists i \mbox{ s.t. }Q_i(t)=0, i\neq LI(t)\}.
\end{equation}
Define the sequence of stopping times $\tau_0,\tau_1,\ldots$ by
\begin{align}\label{eq:sampling times}
    &\tau_0=\inf\{t\geq 1 : \mathbbm{1}_{\Omega_t}=1\}\cr
    &\tau_{k+1}=\inf\{t>\tau_k : \mathbbm{1}_{\Omega_t}=1\}, \quad k=0,1,2\ldots
\end{align}
Since the values of $\tau^A, \tau_0, \tau_1, \ldots$ only depend on the values of $X$, the former are stopping times with respect to both filtrations, ${\cal{F}}_t$ and ${\cal{F}}^{\xi}_t$.
By the definition of the policy, the finiteness of the initial condition and the finiteness of the arrival process, the stopping times $\tau_k$ are finite a.s. for all $k\geq 0$ (the proof follows the same argument as in Lemma 2 of \cite{PI1} and hence is omitted). Define the sampled chain 
\begin{equation}\label{eq:Y}
    Y_k=X({\tau_k}), \quad k\geq 0.
\end{equation}
Denote by ${\cal{A}}^*\subset {\cal{A}}$ the collection of states that can be reached by $Y$, namely
\begin{align*}{\cal{A}}^*=\{\alpha=(\alpha_1,\ldots,\alpha_n,l)\in{\cal{A}}: \alpha_i=0 \text{ for some } i\neq l\}.
\end{align*}
The following lemma asserts that a certain drift criterion is fulfilled by $Y$.
\begin{lemma}\label{lem: drift}
 There exist a constant $\gamma$ such that $0< \gamma\leq1$, a function ${\cal{L}}:{\cal{A}}\rightarrow \mathbbm{R}_+$ of the form
 \begin{equation}\label{eq:L function in lemma}
{\cal{L}}(\alpha)=\sum_{i=1}^n\alpha_i^{1+\gamma}, 
\end{equation}
a positive constant $\epsilon$ and a finite non-empty set $A \subset {\mathcal{A}}$, such that for all  $k\geq 0$ and $\alpha \in {\cal{A}}^*\setminus A$, the following inequality holds:
\begin{align}\label{foster inequalities}
\mathbb{E}_{\alpha}[({\cal{L}}(Y_{k+1})-{\cal{L}}(Y_{k}) )\mathbbm{1}_{\{k<\sigma\}}\mid {\cal{F}}^{\xi}_{\tau_k}]\leq-\epsilon\mathbb{E}_{\alpha}[(\tau_{k+1}-\tau_k) \mathbbm{1}_{\{k<\sigma\}}\mid {\cal{F}}^{\xi}_{\tau_k}], 
\end{align}
where 
\begin{equation}\label{eq:sigma}
    \sigma=\inf \{k \geq 0: Y_k \in A\}.
\end{equation}
\end{lemma}
The stopping time $\sigma$ counts the number of samples needed until $Y$ hits the set $A$ for the first time. Inequality \eqref{foster inequalities} states that, given the initial condition is outside of $A$, as long as $k<\sigma$, the drift of the function $\cal{L}$ applied to $Y$ is strictly negative and is proportional to the expected duration of the intervals between samples. The proof is provided below the proof of the theorem.

Now, fix an initial state $\alpha_0$ such that $X(0)=\alpha_0 \in A$. By the definition of the sampling times $\sigma,\tau_0,\tau_1,\ldots$ in \eqref{eq:sampling times}, we have
\begin{equation}\label{eq:tau_sigma}
    \tau_{\sigma}=\inf\{t\geq 1 : X(t) \in {\cal{A}}^* \cap A\},
\end{equation}
such that $X(\tau_{\sigma})=Y(\sigma)\in {\cal{A}^*}\cap A$. The value of $\tau_{\sigma}$ only depends on the values of $X$, hence $\tau_{\sigma}$ is a stopping time with respect to ${\cal{F}}_t$ and ${\cal{F}}^{\xi}_t$. While $X$ reaches $A$ at time slot $\tau_{\sigma}$, it is not necessarily the first time slot in $[1,\tau_{\sigma}]$ in which this occurs. Since $\tau^A$, defined in \eqref{eq:tau A def}, equals the number of time slots it takes $X$ to reach the set $A$ \textit{for the first time}, we must have $\tau^A\leq \tau_{\sigma}$ and therefore
\begin{equation}\label{eq:tau A less than tau sig}
 \mathbbm{E}_{\alpha_0}[\tau^A]\leq \mathbbm{E}_{\alpha_0}[\tau_{\sigma}].
\end{equation}
Our goal is now to prove that $\mathbbm{E}_{\alpha_0}[\tau_{\sigma}]$ is finite. 
We have:
\begin{align}\label{eq:first tau}
    \mathbbm{E}_{\alpha_0}[\tau_{\sigma}]
    &=\sum_{\alpha \in {\cal{A}^*}\cap A}\mathbbm{E}_{\alpha_0}[\tau_{\sigma} \mid X(\tau_0)=\alpha]\mathbbm{P}_{\alpha_0}(X(\tau_0)=\alpha)\cr
    & \quad+\sum_{\alpha \in {\cal{A}^*}\setminus A}\mathbbm{E}_{\alpha_0}[\tau_{\sigma} \mid X(\tau_0)=\alpha]\mathbbm{P}_{\alpha_0}(X(\tau_0)=\alpha).
\end{align}
%%%%%%%%%%%%%%%%%%%%%%%%%%%%%%%%%%%%%%%%%%%%%%%%5
Conditioned on the event $\{X(\tau_0)=\alpha\in {\cal{A}^*}\cap A\}$, by the definitions of $\tau_0$ in \eqref{eq:sampling times} and $\tau_{\sigma}$ in \eqref{eq:tau_sigma}, we have $\tau_{\sigma}=\tau_0$
and therefore
\begin{equation}\label{eq:tau A hits}
    \mathbbm{E}_{\alpha_0}[\tau_{\sigma} \mid X(\tau_0)=\alpha]=
    \mathbbm{E}_{\alpha_0}[\tau_0 \mid X(\tau_0)=\alpha], \quad \alpha\in {\cal{A}^*}\cap A.
\end{equation}

We proceed with analyzing the term $\mathbbm{E}_{\alpha_0}[\tau_{\sigma} \mid X(\tau_0)=\alpha]$ for the case $\alpha\in {\cal{A}^*}\setminus A$.  
First, we prove that
\begin{equation}\label{eq:SMP}
    \mathbbm{E}_{\alpha_0}[\tau_{\sigma} \mid {\cal{F}}_{\tau_0}]=
    \tau_0+g(X(\tau_0)),
\end{equation}
where $g(x):=\mathbbm{E}_x[\tau_{\sigma}]$. 

%%%%%%%%%%%%%%%%%%%%%%%%%%%%%%%%%%%%%%%%%%%%%
% Start strong markov property proof
Since $\tau_0$ and $\tau_{\sigma}$ are stopping times with respect to ${\cal{F}}_t$ and $\tau_0 \leq \tau_{\sigma}$ a.s., by the Strong Markov Property (Theorem 5.2.5. of \cite{durrett2019probability}), we have
\begin{equation}\label{eq:SMP1}
    \mathbbm{E}_{\alpha_0}[(\tau_{\sigma}-\tau_0)\wedge r \mid {\cal{F}}_{\tau_0}]=\phi^r(X(\tau_0)),
\end{equation}
where $r$ is a positive constant and $\phi^r(x):=\mathbbm{E}_x[\tau_{\sigma}\wedge r]$. By the monotone convergence theorem we have 
\begin{equation}\label{eq:SMP2}
    \mathbbm{E}_x[\tau_{\sigma}\wedge r] \xrightarrow{r \rightarrow \infty}\mathbbm{E}_x[\tau_{\sigma}]
\end{equation}
and
\begin{equation}\label{eq:SMP3}   
    \mathbbm{E}_{\alpha_0}[(\tau_{\sigma}-\tau_0)\wedge r \mid {\cal{F}}_{\tau_0}]\xrightarrow{r \rightarrow \infty} \mathbbm{E}_{\alpha_0}[(\tau_{\sigma}-\tau_0) \mid {\cal{F}}_{\tau_0}].
\end{equation}
Taking the limit w.r.t. $r$ in both sides of \eqref{eq:SMP1} and using \eqref{eq:SMP2} and \eqref{eq:SMP3} yields \eqref{eq:SMP}.

% END strong markov property proof
%%%%%%%%%%%%%%%%%%%%%%%%%%%%%%%%%%%%%%%%%%%%%

Since $X(\tau_0)$ is measurable with respect to ${\cal{F}}_{\tau_0}$, taking the conditional expectation of both sides of \eqref{eq:SMP}, conditioned on $X(\tau_0)$, yields
\begin{equation*}
    \mathbbm{E}_{\alpha_0}[\mathbbm{E}_{\alpha_0}[\tau_{\sigma} \mid {\cal{F}}_{\tau_0}]\mid X(\tau_0)] =
        \mathbbm{E}_{\alpha_0}[\tau_{\sigma} \mid X(\tau_0)]=
    \mathbbm{E}_{\alpha_0}[\tau_0 \mid X(\tau_0)]+g(X(\tau_0)),
\end{equation*}
which implies
\begin{equation}\label{eq:tau A not hits}
        \mathbbm{E}_{\alpha_0}[\tau_{\sigma} \mid X(\tau_0)=\alpha]=
    \mathbbm{E}_{\alpha_0}[\tau_0 \mid X(\tau_0)=\alpha]+\mathbbm{E}_{\alpha}[\tau_{\sigma}].
\end{equation}

Substituting \eqref{eq:tau A hits} and \eqref{eq:tau A not hits} in \eqref{eq:first tau} yields
\begin{align}\label{eq:first tau almost done}
    \mathbbm{E}_{\alpha_0}[\tau_{\sigma}]&=
    \sum_{\alpha \in {\cal{A}^*}}\mathbbm{E}_{\alpha_0}[\tau_0 \mid X(\tau_0)=\alpha]\mathbbm{P}_{\alpha_0}(X(\tau_0)=\alpha)\cr
    & \quad+\sum_{\alpha \in {\cal{A}^*}\setminus A}\mathbbm{E}_{\alpha}[\tau_{\sigma}]\mathbbm{P}_{\alpha_0}(X(\tau_0)=\alpha)\cr
    &= \mathbbm{E}_{\alpha_0}[\tau_0]+\sum_{\alpha \in {\cal{A}^*}\setminus A}\mathbbm{E}_{\alpha}[\tau_{\sigma}]\mathbbm{P}_{\alpha_0}(X(\tau_0)=\alpha),
\end{align}
where the last equality is due to fact that the value of $X(\tau_0)$ is almost surely in ${\cal{A}}^*$. We proceed by showing that the  two terms in the right hand side of \eqref{eq:first tau almost done} are finite, which, combined with \eqref{eq:tau A less than tau sig} proves \eqref{eq:tau} and concludes the proof. 

First, we consider the term $\mathbbm{E}_{\alpha_0}[\tau_0]$.
By \eqref{eq:sampling times}, the stopping time $\tau_0$ is the first time after time 0 such that one of the queue lengths, corresponding to a server different from $LI(0)$, equals zero. Under PI, the servers different from $LI(0)$ do not receive any work during $[1,\tau_0]$. By \eqref{eq:s assumption}, a non-idle server must serve at least one job during a time slot. Thus, 
\begin{equation}\label{eq:bound on tau 0}
    \tau_0 \leq \max_i\{Q_i(0) : i \neq LI(0)\} \leq \max_i\{Q_i(0)\}
\end{equation}
and therefore 
\begin{equation}\label{eq:finite first term}
 \mathbbm{E}_{\alpha_0}[\tau_0]\leq \max_i\{Q_i(0)\}<\infty, 
\end{equation}
where the last inequality is due to the finitness of the initial condition.

Next, we consider the second term in the right hand side of \eqref{eq:first tau almost done}. The proof proceeds by deriving an upper bound on $\mathbbm{E}_{\alpha}[\tau_{\sigma}]$ for $\alpha \in {\cal{A}^*}\setminus A$ by using the sampled chain $Y$, defined in \eqref{eq:Y}, and Lemma \ref{lem: drift}. 
%%%%%%%%%%%%%%%%%%%%%%%%%%%%%%%%%%%%%%
\begin{comment}
By the definition of $\tau_0$ in \eqref{eq:sampling times}, given $X(0)=\alpha \in {\cal{A}^*}\setminus A$, we have $\tau_0=0$ and therefore, by \eqref{eq:Y}, 
\begin{equation}\label{eq:start from the same place}
  Y(0)=X(0)=\alpha, \quad \mbox{if} \quad X(0)=\alpha \in {\cal{A}^*}\setminus A. 
\end{equation}
Namely, $X$ and $Y$ start from the same state at the same time. Now, the stopping time $\sigma$, defined in \eqref{eq:sigma}, equals the number of samples needed until the sampled chain $Y$ hits $A$ for the first time. By the definition of the sampling times $\tau_0,\tau_1,\ldots$ in \eqref{eq:sampling times}, the latter occurs at time slot $\tau_{\sigma}$, namely $X(\tau_{\sigma})=Y(\sigma)\in {\cal{A}^*}\cap A$. While $X$ reaches $A$ at time slot $\tau_{\sigma}$, it is not necessarily the first time slot in $[1,\tau_{\sigma}]$ in which this occurs. Since $\tau^A$, defined in \eqref{eq:tau A def}, equals the number of time slots it takes $X$ to reach the set $A$ \textit{for the first time}, we must have $\tau^A\leq \tau_{\sigma}$ and therefore
\begin{equation}\label{eq:tau A less than tau sig}
 \mathbbm{E}_{\alpha}[\tau^A]\leq \mathbbm{E}_{\alpha}[\tau_{\sigma}], \quad \alpha \in {\cal{A}^*}\setminus A.
\end{equation}
We turn to bound $\mathbbm{E}_{\alpha_0}[\tau_{\sigma}]$. 
\end{comment}
%%%%%%%%%%%%%%%%%%%%%%%%%%%%%%%%%%%%%%
Taking expectation on both sides of \eqref{foster inequalities} yields
\begin{equation}\label{eq:taking expectation}
\mathbb{E}_{\alpha}[({\cal{L}}(Y_{k+1})-{\cal{L}}(Y_{k}) )\mathbbm{1}_{\{k<\sigma\}}]\leq-\epsilon\mathbb{E}_{\alpha}[(\tau_{k+1}-\tau_k) \mathbbm{1}_{\{k<\sigma\}}]. 
\end{equation}
We use \eqref{eq:taking expectation} twice: to prove that $\tau_{\sigma}$ is well defined, namely, that $\sigma<\infty$ a.s., and then to derive a bound for $\mathbbm{E}_{\alpha}[\tau_{\sigma}]$.  

%%% proving sigma<infty %%%%%%%%%%%%%%%
By the definition of the sampling times in \eqref{eq:sampling times}, we have $\tau_{k+1}-\tau_k \geq 1$. Thus \eqref{eq:taking expectation} implies that 
\begin{equation*}
\mathbb{E}_{\alpha}[({\cal{L}}(Y_{k+1})-{\cal{L}}(Y_{k}) )\mathbbm{1}_{\{k<\sigma\}}]\leq-\epsilon\mathbb{E}_{\alpha}[\mathbbm{1}_{\{k<\sigma\}}]. 
\end{equation*}
Summing over $k \in [0,m-1]$ for $m \geq 1$ yields
\begin{equation*}
\E_{\alpha}[{\cal{L}}(Y_{m \wedge \sigma})]-\E_{\alpha}[{\cal{L}}(Y_0)]\leq
-\epsilon\sum_{k=0}^{m-1}\E_{\alpha}[\mathbbm{1}_{\{k<\sigma\}}]=-\epsilon\sum_{k=0}^{m-1} \mathbbm{P}_{\alpha}(\sigma>k),
\end{equation*}
After rearranging we obtain
\begin{equation*}
\sum_{k=0}^{m-1} \mathbbm{P}_{\alpha}(\sigma>k)\leq
\E_{\alpha}[{\cal{L}}(Y_0)]/\epsilon-\E_{\alpha}[{\cal{L}}(Y_{m \wedge \sigma})]/\epsilon\leq \E_{\alpha}[{\cal{L}}(Y_0)]/\epsilon=\E_{\alpha}[{\cal{L}}(X(\tau_0))]/\epsilon,
\end{equation*}
where in the second inequality we used the fact that $\E_{\alpha}[{\cal{L}}(Y_{m \wedge \sigma})]\geq0$.
Hence
\begin{equation}\label{eq:expect of sigma}
\E_{\alpha}[\sigma]=\lim_{m \rightarrow \infty}\sum_{k=0}^{m-1} \mathbbm{P}_{\alpha}(\sigma>k)\leq \E_{\alpha}[{\cal{L}}(X(\tau_0))]/\epsilon.
\end{equation}
We now argue that $\E_{\alpha}[{\cal{L}}(X(\tau_0))]$ is finite for all initial states $\alpha \in {\cal{A}}$ (and therefore, specifically, for $\alpha \in {\cal{A}}^* \setminus A$). Using the definition of $\cal{L}$ in \eqref{eq:L function in lemma}, we have
\begin{equation}\label{eq:using def of L}
\mathbbm{E}_{\alpha}[{\cal{L}}(X(\tau_0))]=\mathbbm{E}_{\alpha}\Big[\sum_{i=1}^n(Q_i(\tau_0))^{1+\gamma}\Big]\leq \mathbbm{E}_{\alpha}\Big[\sum_{i=1}^n(Q_i(\tau_0))^{2}\Big],
\end{equation}
where the last inequality is due to the fact that $0< \gamma \leq 1$.
Denote the maximum queue length at time $t$ by $Q_{M}(t)$, namely, $Q_{M}(t)=\max_i\{Q_i(t)\}$. The value of $Q_{M}(0)$ is finite for all initial states $\alpha \in {\cal{A}}$. The quantities $Q_i(\tau_0)$ for servers different from the Last Idle server are each bounded from above by $Q_{M}(0)$. The queue length of the Last Idle server at $\tau_0$ is bounded from above by $Q_{M}(0)$ plus the number of jobs that arrived to the system during $[1,\tau_0]$. Therefore, we have
\begin{align}\label{eq:qmax bound}
\mathbbm{E}_{\alpha}\Big[\sum_{i=1}^n(Q_i(\tau_0))^{2}\Big]&\leq \mathbbm{E}_{\alpha}\Bigg[(n-1)Q_{M}^2(0)+\Big(Q_{M}(0)+\sum_{t=1}^{\tau_0}a(t)\Big)^{2}\Bigg]\cr
&\leq \mathbbm{E}_{\alpha}\Bigg[(n-1)Q_{M}^2(0)+\Big(Q_{M}(0)+\sum_{t=1}^{Q_{M}(0)}a(t)\Big)^{2}\Bigg]\cr
&\leq B\mathbbm{E}_{\alpha}\Big[Q_{M}^2(0)\Big]=B \Big(\max_{i \in [n]}\alpha_i \Big) ^2,
\end{align}
for some constant $B>0$, where the second inequality is due to \eqref{eq:bound on tau 0} and the third inequality is due to the finite first and second moments of the arrival process and the fact that $Q_{M}(0)$ is an integer. Using \eqref{eq:qmax bound} in \eqref{eq:using def of L} and recalling that $Q_{max}$ is finite yields
\begin{equation}\label{eq:finiteness of second term}
\mathbbm{E}_{\alpha}[{\cal{L}}(X(\tau_0))]\leq B \Big(\max_{i \in [n]}\alpha_i \Big) ^2<\infty, \quad \alpha \in {\cal{A}}.
\end{equation}
Equations \eqref{eq:expect of sigma} and \eqref{eq:finiteness of second term} establish that $\mathbbm{E}_{\alpha}[\sigma]< \infty$ and consequently 
\begin{equation}\label{eq:sigma less than infty}
    \sigma<\infty \quad \mbox{a.s.}
\end{equation}
%%%%%%%%%%%%%%%%%%%%%%%%%%%%%%%%%%%%%%%%
Now, revisiting \eqref{eq:taking expectation}, summing over $k \in [0,m-1]$ yields
\begin{equation*}
\E_{\alpha}[{\cal{L}}(Y_{m \wedge \sigma})]-\E_{\alpha}[{\cal{L}}(Y_{0})]\leq
-\epsilon\E_{\alpha}[\tau_{m \wedge \sigma}].
\end{equation*}
Using the fact that $\E_{\alpha}[{\cal{L}}(Y_{m \wedge \sigma})]\geq0$ and rearranging yields
\begin{equation}\label{eq:bound on m}
\E_{\alpha}[\tau_{m \wedge \sigma}]\leq \E_{\alpha}[{\cal{L}}(X(\tau_0))]/\epsilon.
\end{equation}
Using \eqref{eq:sigma less than infty}, the monotone convergence theorem and \eqref{eq:bound on m}, we obtain
\begin{equation}\label{eq:n sigma}
\E_{\alpha}[\tau_{\sigma}]=\E_{\alpha}[\lim_{m \rightarrow \infty}\tau_{m \wedge\sigma}]
=\lim_{m \rightarrow \infty}\E_{\alpha}[\tau_{m \wedge\sigma}]
\leq \E_{\alpha}[{\cal{L}}(X(\tau_0))]/\epsilon\leq \bar{B}\Big(\max_{i \in [n]}\alpha_i \Big) ^2,
\end{equation}
where $\bar{B}>0$ is some constant and in the last inequality we used \eqref{eq:finiteness of second term}.
Using \eqref{eq:tau A less than tau sig}, \eqref{eq:first tau almost done} and  \eqref{eq:n sigma} yields
\begin{align}\label{eq:first tau almost done 2}
    \mathbbm{E}_{\alpha_0}[\tau^A]\leq\mathbbm{E}_{\alpha_0}[\tau_{\sigma}]
    &\leq \mathbbm{E}_{\alpha_0}[\tau_0]+\bar{B}\sum_{\alpha \in {\cal{A}^*}\setminus A} \Big(\max_{i \in [n]}\alpha_i \Big) ^2\mathbbm{P}_{\alpha_0}(X(\tau_0)=\alpha)
\end{align} 
Recalling the definition of $Q_{M}$, the second term on the right hand side of \eqref{eq:first tau almost done 2} satisfies
\begin{equation*}
    \sum_{\alpha \in {\cal{A}^*}\setminus A}\Big(\max_{i \in [n]}\alpha_i \Big) ^2\mathbbm{P}_{\alpha_0}(X(\tau_0)=\alpha)\leq \sum_{\alpha \in {\cal{A}^*}}\Big(\max_{i \in [n]}\alpha_i \Big) ^2\mathbbm{P}_{\alpha_0}(X(\tau_0)=\alpha)=\mathbbm{E}_{\alpha_0}\Big[Q_M^2(\tau_0)\Big].
\end{equation*}
Therefore
\begin{equation}\label{eq:only 2 left}
\mathbbm{E}_{\alpha_0}[\tau^A]\leq \mathbbm{E}_{\alpha_0}[\tau_0]+\bar{B}\mathbbm{E}_{\alpha_0}\Big[Q_M^2(\tau_0)\Big].
\end{equation}

Combining \eqref{eq:finite first term}, \eqref{eq:finiteness of second term} and \eqref{eq:only 2 left} proves \eqref{eq:tau},
which concludes the proof. 
%\SV{I think we may make it easier to follow if we add 2 figures to illustrate what is going on here...}
\qed\\

\noindent \textbf{Proof of Lemma \ref{lem: drift}.}

First, we discuss the term $\mathbbm{1}_{\{k<\sigma\}}$ in \eqref{foster inequalities}. By \eqref{eq:Y}, we have $(Y_0,\ldots,Y_k)=(X({\tau_0}),\ldots,X({\tau_k}))$. By \eqref{eq:filtration}, ${\cal{F}}^{\xi}_{\tau_k}$ contains the entire evolution of the process $X$ during $[0,\tau_k]$ and therefore the values of the sampled chain $Y$, $Y_0,\ldots,Y_k$, are measurable with respect to ${\cal{F}}^{\xi}_{\tau_k}$. Since $\sigma$ (defined in \eqref{eq:sigma}) corresponds to the first sample which lies in the finite set $A$, the values $Y_0,\ldots,Y_k$ are enough to determine whether or not the event $\{k<\sigma\}$ occurred. Thus, the random variable $\mathbbm{1}_{\{k<\sigma\}}$ is also measurable with respect to ${\cal{F}}^{\xi}_{\tau_k}$, and we can write \eqref{foster inequalities} as
\begin{align}\label{foster inequalities out}
\mathbbm{1}_{\{k<\sigma\}}\mathbb{E}_{\alpha}[({\cal{L}}(Y_{k+1})-{\cal{L}}(Y_{k}) )\mid {\cal{F}}^{\xi}_{\tau_k}]\leq-\epsilon\mathbbm{1}_{\{k<\sigma\}}\mathbb{E}_{\alpha}[(\tau_{k+1}-\tau_k) \mid {\cal{F}}^{\xi}_{\tau_k}], 
\end{align}
Clearly, on the event $\{k\geq\sigma\}$, the two sides of inequality \eqref{foster inequalities out} equal zero and \eqref{foster inequalities} holds. 
The rest of proof is dedicated to proving \eqref{foster inequalities} on the event $\{k<\sigma\}$, namely, given that the values of $Y_0,\ldots,Y_k$ are all outside of $A$. For simplicity, we drop the subscript $\alpha$ from the expected value notation and write $\mathbbm{E}$ for $\mathbbm{E}_\alpha$ for the rest of the proof.

We begin by fixing $\epsilon>0$ and a function $\cal{L}$. We then consider a finite set $A$ of the form
\begin{equation}\label{eq:finite_set}
A=\{\alpha=(\alpha_1,\ldots,\alpha_n,l)\in {\cal{A}} \mid \sum_i \alpha_i \leq \cal{C}\},
\end{equation}
and show that if $\cal{C}$ is large enough, \eqref{foster inequalities} holds. 

Since $\lambda<\sum_{i=1}^n \mu_i$, there exists $\epsilon_0>0$ such that 
\begin{equation*}\label{eq:epsilon_0}
    \lambda-\sum_{i=1}^n \mu_i=-\epsilon_0.
\end{equation*}
Denote $\mu_{min}=\min_{i}\mu_i$ and let
\begin{equation}\label{eq:epsilon}
    \epsilon=(\epsilon_0\wedge\mu_{min})/4.
\end{equation}
Fix $0<\gamma\leq 1$ such that
\begin{equation}\label{eq:gamma}
    \max_{j\in \{1,\ldots,n\}}\Big\{\big([\lambda-\mu_j]^+\big)^{1+\gamma}-\sum_{i\neq j}\mu_i\Big\}<-2\epsilon.
\end{equation}
To prove the existence of such a $\gamma$, define $g_j(\gamma)=([\lambda-\mu_j]^+)^{1+\gamma} -\sum_{i \neq j}\mu_i$. Now,
\begin{align}\label{eq:g}
g_j(0)=[\lambda-\mu_j]^+ -\sum_{i \neq j}\mu_i=-\big(\epsilon_0\wedge\sum_{i \neq j}\mu_i\big)
\leq -\big(\epsilon_0\wedge\mu_{min}\big)=-4\epsilon.
\end{align}
Define $h(\gamma)=\max_{j\in \{1,\ldots,n\}}g_j(\gamma)$. By \eqref{eq:g}, $h(0)\leq-4\epsilon$. Since $g_j(\gamma)$ are continuous functions of $\gamma$, so is $h(\gamma)$. Thus, there exists $\gamma\in (0,1]$ such that $h(\gamma)< -2\epsilon$, and \eqref{eq:gamma} holds.

Recall (by \eqref{eq:L function in lemma}) the function ${\cal{L}}:{\cal{A}}\rightarrow \mathbbm{R}_+$, for $\alpha=(\alpha_1,\ldots,\alpha_n,l)\in {\cal{A}}$, is defined by
\begin{equation}\label{eq:L function}
{\cal{L}}(\alpha)=\sum_{i=1}^n\alpha_i^{1+\gamma}. 
\end{equation}
Starting with the left hand side of \eqref{foster inequalities}, we use \eqref{eq:L function} to obtain
\begin{align}\label{eq:1}
\mathbb{E}[{\cal{L}}(Y_{k+1})-{\cal{L}}(Y_{k}) \mid {\cal{F}}^{\xi}_{\tau_k}]=\mathbb{E}\Big[\sum_{i=1}^n\Big( Q_i^{1+\gamma}(\tau_{k+1})-{Q_i^{1+\gamma}(\tau_k)}\Big)\mid {\cal{F}}^{\xi}_{\tau_k}\Big].
\end{align}
The proof proceeds by decomposing the sum in the right hand side of \eqref{eq:1} into two parts and analyzing them separately: the first part corresponds to the Last Idle server at time $\tau_{k}+1$, which we denote by
\begin{equation}\label{eq:def of lk}
    l_k=LI(\tau_k+1).
\end{equation}
The second part corresponds to the rest of the servers. 

By \eqref{eq:maxflow}, the value of $l_k$ is determined using $X(\tau_k)$ and $\xi_{\tau_k+1}$, which, by \eqref{eq:filtration}, are included in ${\cal{F}}^{\xi}_{\tau_k}$. Thus, the random variable $l_k$ is measurable with respect to ${\cal{F}}^{\xi}_{\tau_k}$. By the definition of the sampling times in \eqref{eq:omega} and \eqref{eq:sampling times}, the identity of the Last Idle server remains unchanged during $[\tau_k+1,\tau_{k+1}]$ and equals to $l_k$. Also, 
\begin{equation}\label{eq:starts from zero}
    Q_{l_k}(\tau_k)=0. 
\end{equation}
Thus, during $[\tau_k+1,\tau_{k+1}]$, the server corresponding to $l_k$ receives all incoming jobs and its queue length is given by a reflected random walk starting at zero. The rest of the queue lengths decrease until the first reaches zero at time $\tau_{k+1}$. Therefore, using the balance equation \eqref{eq:maxflow}, we obtain
\begin{equation}\label{eq:dynamics}
    Q_{i}(\tau_{k+1})=\Bigg[Q_i(\tau_k)-\sum_{t=\tau_k+1}^{\tau_{k+1}}s_i(t)\Bigg]^+, \quad i \in \{1,\ldots,n\}\setminus \{l_k\}.
\end{equation}
For simplicity, denote 
\begin{align}\label{eq:easy notation}
    &Q_i^k=Q_i(\tau_k), \quad i \in \{1,\ldots,n\},\cr
    &R_i^k=\sum_{t=\tau_k+1}^{\tau_{k+1}}s_i(t), \quad i \in \{1,\ldots,n\}\setminus \{l_k\}.
\end{align}
Using \eqref{eq:def of lk}, \eqref{eq:starts from zero}, \eqref{eq:dynamics} and \eqref{eq:easy notation}, we decompose the sum in the right hand side of \eqref{eq:1} as follows: 
\begin{align}\label{eq:2}
\mathbb{E}\Big[\sum_{i=1}^n\Big( Q_i^{1+\gamma}(\tau_{k+1})-{Q_i^{1+\gamma}(\tau_k)}\Big)\mid {\cal{F}}^{\xi}_{\tau_k}\Big]
&=\mathbb{E}\Big[\sum_{i\neq l_k}^n\Big( ([Q_i^k-R_i^k]^+)^{1+\gamma}-{(Q_i^k)^{1+\gamma}}\Big)\mid {\cal{F}}^{\xi}_{\tau_k}\Big]\cr  &\quad +\mathbb{E}\Big[ \Big({Q_{l_k}(\tau_{k+1})}\Big)^{1+\gamma}\mid {\cal{F}}^{\xi}_{\tau_k}\Big].
\end{align}
Denote
\begin{equation}\label{eq:delta}
    \Delta_k=\tau_{k+1}-\tau_k.
\end{equation}
We proceed by bounding the two terms in the right hand side of \eqref{eq:2} separately. We begin with $\mathbb{E}\Big[ \Big({Q_{l_k}(\tau_{k+1})}\Big)^{1+\gamma}\mid {\cal{F}}^{\xi}_{\tau_k}\Big]$. Conditioned on ${\cal{F}}^{\xi}_{\tau_k}$, the duration of the interval $[\tau_k+1,\tau_{k+1}]$, given by $\Delta_k$, depends only on the queue lengths at the servers other than $l_k$ and is therefore independent of the reflected random walk in server $l_k$. Denote 
\begin{align}\label{eq: beta and sigma}
    \beta_{i}&:=\lambda-\mu_{i}\cr
    \sigma_{i}^2&:=\sigma_a^2+\sigma_{s_{i}}^2.
\end{align}
Using a modification of Lemma 6 of \cite{PI1}, we obtain
\begin{equation}\label{eq: RRW bound}
    \mathbb{E}\Big[ \Big({Q_{l_k}(\tau_{k+1})}\Big)^{1+\gamma}\mid {\cal{F}}^{\xi}_{\tau_k}\Big]\leq ([\lambda-\mu_{l_k}]^+)^{1+\gamma}\mathbb{E}[\Delta_k^{1+\gamma}\mid {\cal{F}}^{\xi}_{\tau_k}]+
    C\mathbb{E}[\Delta_k^{0.75(1+\gamma)}\mid {\cal{F}}^{\xi}_{\tau_k}],
\end{equation}
where $C=\max_i\{\big(16\sigma_i+8(\sigma_i)
^{1/2}(\beta_i)^+\big)^{(1+\gamma)/2}\}$. The proof is deferred to the Appendix. 

Next, we consider the first term in the right hand side of \eqref{eq:2}. For ease of exposition, we suppress the index $k$ and write $Q_i$ and $R_i$ for $Q_i^k$ and $R_i^k$ respectively. We have 
\begin{align}\label{eq:taylor}
    ([Q_i-R_i]^+)^{1+\gamma}-{Q_i^{1+\gamma}}&=Q_i^{1+\gamma}(([1-R_i/(Q_i\vee 1)]^+)^{1+\gamma}-1)\cr
    & \quad \leq Q_i^{1+\gamma}([1-R_i/(Q_i\vee 1)]^+-1)\cr
    &= -Q_i^{\gamma}(Q_i\wedge R_i),
\end{align}
where the inequality is due to the fact that $R_i/(Q_i\vee 1) \geq 0$ and thus $0\leq[1-R_i/(Q_i\vee 1)]^+\leq 1$. 
Combining \eqref{eq:1}, \eqref{eq:2}, \eqref{eq: RRW bound} and \eqref{eq:taylor} we obtain
\begin{align}\label{eq:before cases}
    \mathbb{E}[{\cal{L}}(Y_{k+1})-{\cal{L}}(Y_{k}) \mid {\cal{F}}^{\xi}_{\tau_k}]\leq 
    -\epsilon\mathbb{E}[\Delta_k \mid {\cal{F}}^{\xi}_{\tau_k}]+\Phi_k,
\end{align}
where
\begin{equation}\label{eq:phi}
    \Phi_k:=    
    \mathbb{E}\Big[\epsilon\Delta_k+([\lambda-\mu_{l_k}]^+)^{1+\gamma}\Delta_k^{1+\gamma}+
    C\Delta_k^{0.75(1+\gamma)}-\sum_{i\neq l_k}^nQ_i^{\gamma}(Q_i\wedge R_i)\mid {\cal{F}}^{\xi}_{\tau_k}\Big].
\end{equation}
Thus, using \eqref{eq:before cases}, \eqref{foster inequalities} will follow once we prove that $\Phi_k\leq 0$ a.s. To this end, denote
\begin{align*}
    &Q_{min}=\min_{1\leq i \leq n, i \neq l_k}Q_i, \cr
    &Q_{max}=\max_{1\leq i \leq n}Q_i.
\end{align*}
Denote by $i_{max}$ the server with the maximum queue length at time $\tau_k$, namely  
\begin{equation*}
    i_{max}=\mbox{argmax}_{1\leq i \leq n}\{Q_i\}
\end{equation*}
Since the queue length at time $\tau_k$ of the Last Idle server $l_k$ equals zero, we have $i_{max}\neq l_k$.

The key to analyzing \eqref{eq:phi} is to relate $\Delta_k$, the duration of the interval $[\tau_{k}+1,\tau_{k+1}]$, to the value of $Q_{min}$, which, by \eqref{eq:filtration}, is measurable with respect to ${\cal{F}}^{\xi}_{\tau_k}$. By the definition of the sampling times in \eqref{eq:sampling times} we have $\Delta_k\geq 1$. By \eqref{eq:s assumption}, any strictly positive queue lengths of the servers different from $l_k$ must decrease by at least one and by at most  $s_{max}$ at each time slot. Therefore, given ${\cal{F}}^{\xi}_{\tau_k}$, we have
\begin{align}\label{eq:relate}
    (Q_{min}/s_{max})\vee 1 \leq \Delta_k\leq Q_{min}\vee 1
\end{align}

The argument proceeds by analyzing \eqref{eq:phi} in two cases, corresponding to the value of $Q_{min}$. The idea is that if $Q_{min}$ is small, then by \eqref{eq:relate}, so is $\Delta_k$. Thus, if the finite set in \eqref{eq:finite_set} is chosen large enough so that $Q_{max}$ must be very large, one of the members of $\sum_{i\neq l_k}^nQ_i^{\gamma}(Q_i\wedge R_i)$ is very large compared to the other terms in the right hand side of \eqref{eq:phi}, resulting in $\Phi_k \leq 0$. Otherwise, if $Q_{min}$ is large, then by \eqref{eq:relate} so is $\Delta_k$. In this case, we prove that  
$$\mathbb{E}\Big[\sum_{i\neq l_k}^nQ_i^{\gamma}(Q_i\wedge S_i)\mid {\cal{F}}^{\xi}_{\tau_k}\Big]\geq 
\mathbb{E}\Big[\Big(\sum_{i\neq l_k}^n \mu_i\Big)\Delta_k^{1+\gamma}-(n-1)s_{max}\Delta_k^\gamma\mid {\cal{F}}^{\xi}_{\tau_k}\Big].$$
Thus, if we use this to further bound $\Phi_k$ from above, by \eqref{eq:gamma}, the coefficient of the leading term $\Delta_k^{1+\gamma}$ is negative. If $\Delta_k$ is large enough we again obtain $\Phi_k \leq 0$.

Define
\begin{equation}\label{eq:f}
  f(x)=  \epsilon x -2\epsilon x^{1+\gamma}+Cx^{0.75(1+\gamma)}+(n-1)s_{max}x^\gamma,
\end{equation}
where $C$ is from \eqref{eq: RRW bound}.
Since the coefficient of the leading term $x^{1+\gamma}$ in $f(x)$ is negative, there exists $u>0$ such that 
\begin{equation}\label{eq:u}
   u>s^2_{max} \quad \mbox{ and } \quad f(x)<0 \quad \mbox{ if } x>u/s_{max}.
\end{equation}
Fix such $u$.\\
Choose the constant $\cal{C}$ in the definition of the finite set $A$ in \eqref{eq:finite_set} such that
\begin{equation}\label{eq:choice of finite set}
    {\cal{C}}>ns_{max}u + n\Big(\epsilon u+(\max_i\{[\lambda-\mu_i]^+\})^{1+\gamma}u^{1+\gamma}+
    Cu^{0.75(1+\gamma)}\Big)^{1/\gamma}+n\Big(\epsilon+n(\sigma_a^2+\lambda^2)\Big)^{1/\gamma}.
\end{equation}
By \eqref{eq:finite_set}, for any state outside of $A$ we have ${\cal{C}}<\sum_iQ_i\leq nQ_{max}$. Thus, by \eqref{eq:choice of finite set}, we obtain
\begin{equation}\label{eq:qmax}
    Q_{max}>s_{max}u + \Big(\epsilon u+(\max_i\{[\lambda-\mu_i]^+\})^{1+\gamma}u^{1+\gamma}+
    Cu^{0.75(1+\gamma)}\Big)^{1/\gamma}.
\end{equation}
We now consider two cases that correspond to the value of $Q_{min}$, which, by \eqref{eq:relate}, is related to the duration of the interval. The first two members of the right hand side of \eqref{eq:choice of finite set} correspond to our requirements of the finite set $A$ for dealing with the aforementioned two cases. The third member is used for the proof of Theorem \eqref{main result PI-Split} on PI-Split.\\
%%%%%%%%%%%%%%%%%%%%%%%%%%%%%%%%%%%%%%%%%%%%%%%%%%%%%%%
%%%% case 1

\noindent \textbf{Case 1: (short interval)} $Q_{min}\leq u$. 

In this case, by using \eqref{eq:relate}, we obtain 
\begin{equation}\label{eq:case 1 delta}
    \Delta_k\leq u \quad \mbox{if} \quad Q_{min}\leq u,
\end{equation}
namely, the duration of the interval is at most $u$.
Since by \eqref{eq:qmax} $Q_{max}>s_{max}u$ and $s_{max}$ is the maximum number of jobs that can be served during a single time slot, the largest queue length at $\tau_k$ (in server $i_{max}$) does not reach zero at $\tau_{k+1}$. By \eqref{eq:s assumption} the queue length in server $i_{max}$ must decrease by at least one. Thus, using \eqref{eq:easy notation}, we conclude that $R_{i_{max}}\geq 1$ and $Q_{max}\wedge R_{i_{max}}\geq 1$.
 Therefore, we have
\begin{equation}\label{eq:case 1 qmax}
    \sum_{i\neq l_k}^nQ_i^{\gamma}(Q_i\wedge R_i)\geq Q_{max}^{\gamma}(Q_{max}\wedge R_{i_{max}})\geq Q_{max}^\gamma.
\end{equation}
Using \eqref{eq:case 1 delta} and \eqref{eq:case 1 qmax} to upper bound the right hand side of \eqref{eq:phi} yields
$$\Phi_k\leq \mathbb{E}\Big[\epsilon u+([\lambda-\mu_{l_k}]^+)^{1+\gamma}u^{1+\gamma}+
    Cu^{0.75(1+\gamma)}-Q_{max}^\gamma\mid {\cal{F}}^{\xi}_{\tau_k}\Big]<0,$$
where the last inequality is due to \eqref{eq:qmax}.\\
%%%%%%%%%%%%%%%%%%%%%%%%%%%%%%%%%%%%%%%%%%%%%%%%%%%%%%%
%%%% case 2

\noindent \textbf{Case 2: (long interval)} $Q_{min}> u$. 

By the definition of the sampling times, $R_i$ can be larger than $Q_i$ by at most $s_i(\tau_{k+1})$ (the number of jobs potentially served by server $i$ at the last time slot of the interval $[\tau_k+1,\tau_{k+1}]$), which we assumed in \eqref{eq:s assumption} is bounded from above by $s_{max}$. Thus
\begin{equation*}
    Q_i\wedge R_i\geq (R_i-s_{max})\wedge R_i=R_i-s_{max}.
\end{equation*}
Therefore
\begin{equation}\label{eq:case2bound}
    \mathbb{E}\Big[\sum_{i\neq l_k}^nQ_i^{\gamma}(Q_i\wedge R_i)\mid {\cal{F}}^{\xi}_{\tau_k}\Big]\geq \mathbb{E}\Big[\sum_{i\neq l_k}^nQ_i^{\gamma}(R_i-s_{max})\mid {\cal{F}}^{\xi}_{\tau_k}\Big]=\sum_{i\neq l_k}^nQ_i^{\gamma}(\mathbb{E}[R_i\mid {\cal{F}}^{\xi}_{\tau_k}]-s_{max}),
\end{equation}
where the last equality is due to the fact that $Q_i$ is measurable with respect to ${\cal{F}}^{\xi}_{\tau_k}$ for all $i$. 

Next, we calculate $\mathbb{E}[R_i\mid {\cal{F}}^{\xi}_{\tau_k}]$. By the definitions of $R_i$ in \eqref{eq:easy notation} and $\Delta_k$ in \eqref{eq:delta}, this is the expected value of a random walk with strictly positive step sizes, distributed as $s_i(0)$, starting at time $\tau_k+1$ and stopped after $\Delta_k$ time slots. If $\Delta_k$ only depended on the queue length at server $i$, then it would have been a stopping time, and by Wald's Equation and the finitness of the sampling times the expected value would have been equal to $\mu_i\mathbb{E}[\Delta_k\mid {\cal{F}}^{\xi}_{\tau_k}]$. However, this is not the case here, since $\Delta_k$ depends on all servers different from $l_k$. Even so, we are able to prove the same result holds. Using  the definition of $R_i$ in \eqref{eq:easy notation}  and the fact that $\tau_{k+1}=\tau_k+\Delta_k$ we obtain
\begin{align}\label{eq:s_i1}
    \mathbb{E}[R_i\mid {\cal{F}}^{\xi}_{\tau_k}]&=\mathbb{E}\Bigg[\sum_{t=\tau_k+1}^{\tau_{k+1}}s_i(t)\mid {\cal{F}}^{\xi}_{\tau_k}\Bigg]=
    \mathbb{E}\Bigg[\sum_{t=\tau_k+1}^{\tau_k+\Delta_k}s_i(t)\mid {\cal{F}}^{\xi}_{\tau_k}\Bigg]\cr
    &=\mathbb{E}\Bigg[\sum_{t=\tau_k+1}^{\tau_k+Q_{min}}s_i(t)\mathbbm{1}_{\{t\leq \tau_k+\Delta_k\}}\mid {\cal{F}}^{\xi}_{\tau_k}\Bigg]\cr
    &=\mathbb{E}\Bigg[\sum_{m=1}^{Q_{min}}s_i(\tau_k+m)\mathbbm{1}_{\{m\leq \Delta_k\}}\mid {\cal{F}}^{\xi}_{\tau_k}\Bigg],
\end{align}
where the second to last equality is due to the fact that by \eqref{eq:relate}, we have $\Delta_k\leq Q_{min}$, and the last equality is a simple variable change.

Since the random variables $Q_{min}$ and $\tau_k$ are both measurable with respect to ${\cal{F}}^{\xi}_{\tau_k}$ 
%\GM{is this true for $\tau_k$? make sure}
and are finite a.s., we can change the order of the summation and the expectation, yielding
\begin{align}\label{eq:s_i2}
    \mathbb{E}\Bigg[\sum_{m=1}^{Q_{min}}s_i(\tau_k+m)\mathbbm{1}_{\{m\leq \Delta_k\}}\mid {\cal{F}}^{\xi}_{\tau_k}\Bigg]=
    \sum_{m=1}^{Q_{min}}\mathbb{E}[s_i(\tau_k+m)\mathbbm{1}_{\{m\leq \Delta_k\}}\mid {\cal{F}}^{\xi}_{\tau_k}].
\end{align}

The event $\{m\leq \Delta_k\}$ is equivalent to the event $\{\Delta_k \leq m-1\}^c$.  Thus, given ${\cal{F}}^{\xi}_{\tau_k}$, whether or not the event $\{m\leq \Delta_k\}$ occurred is determined by the values of all the service processes, in all servers, during the interval $[\tau_k+1,\ldots,\tau_k+m-1]$, which are independent of $s_i(\tau_k+m)$. Therefore
\begin{align}\label{eq:s_i3}
    \sum_{m=1}^{Q_{min}}\mathbb{E}[s_i(\tau_k+m)\mathbbm{1}_{\{m\leq \Delta_k\}}\mid {\cal{F}}^{\xi}_{\tau_k}]&=
    \sum_{m=1}^{Q_{min}}\mathbb{E}[s_i(\tau_k+m)\mid {\cal{F}}^{\xi}_{\tau_k}]\mathbb{E}[\mathbbm{1}_{\{m\leq \Delta_k\}}\mid {\cal{F}}^{\xi}_{\tau_k}]\cr
    &=\sum_{m=1}^{Q_{min}}\mu_i\mathbb{E}[\mathbbm{1}_{\{m\leq \Delta_k\}}\mid {\cal{F}}^{\xi}_{\tau_k}]=\mu_i\mathbb{E}[\Delta_k\mid {\cal{F}}^{\xi}_{\tau_k}],
\end{align}
where the second equality is due to $s_i(\tau_k+m)$ being independent of ${\cal{F}}^{\xi}_{\tau_k}$ and the last equality is due to the fact that $\sum_{m=1}^{Q_{min}}\mathbbm{1}_{\{m\leq \Delta_k\}}=\Delta_k$.

Combining \eqref{eq:case2bound}, \eqref{eq:s_i1}, \eqref{eq:s_i2} and \eqref{eq:s_i3} yields
\begin{equation}\label{eq:case2bound2}
    \mathbb{E}\Big[\sum_{i\neq l_k}^nQ_i^{\gamma}(Q_i\wedge R_i)\mid {\cal{F}}^{\xi}_{\tau_k}\Big]\geq \sum_{i\neq l_k}^nQ_i^{\gamma}(\mu_i\mathbb{E}[\Delta_k\mid {\cal{F}}^{\xi}_{\tau_k}]-s_{max})=\sum_{i\neq l_k}^n \mathbb{E}[Q_i^{\gamma}(\mu_i\Delta_k-s_{max})\mid {\cal{F}}^{\xi}_{\tau_k}].
\end{equation}
Since we consider the case where $Q_{min}>u$, and by \eqref{eq:u} $u>s^2_{max}$, by \eqref{eq:relate} we obtain
\begin{equation}\label{eq:finally}
    \Delta_k > Q_{min}/s_{max}>u/s_{max}>s_{max}.
\end{equation}
Thus, using the fact that by \eqref{eq:mu assumption}, $\mu_i\geq 1$ for all $i$, we obtain
\begin{equation}\label{eq:positive}
    \mu_i\Delta_k-s_{max}>0.
\end{equation}
Using \eqref{eq:case2bound2}, \eqref{eq:positive} and the fact that $\Delta_k \leq Q_{min} \leq Q_i$ for $i \neq l_k$, we obtain
\begin{align}\label{eq:case2bound3}
    \mathbb{E}\Big[\sum_{i\neq l_k}^nQ_i^{\gamma}(Q_i\wedge R_i)\mid {\cal{F}}^{\xi}_{\tau_k}\Big]&\geq \sum_{i\neq l_k}^n \mathbb{E}[\Delta_k^{\gamma}(\mu_i\Delta_k-s_{max})\mid {\cal{F}}^{\xi}_{\tau_k}]=\sum_{i\neq l_k}^n \mathbb{E}[\mu_i\Delta_k^{1+\gamma}-s_{max}\Delta_k^{\gamma}\mid {\cal{F}}^{\xi}_{\tau_k}]\cr
    &= \mathbb{E}\Big[\Big(\sum_{i\neq l_k}^n\mu_i\Big)\Delta_k^{1+\gamma}\mid {\cal{F}}^{\xi}_{\tau_k}\Big]-(n-1)s_{max}\mathbb{E}[\Delta_k^{\gamma}\mid {\cal{F}}^{\xi}_{\tau_k}].
\end{align}
Using \eqref{eq:case2bound3} to bound the right hand side of \eqref{eq:phi} from above, we obtain
\begin{align*}
    \Phi_k &\leq \mathbb{E}\Big[\epsilon\Delta_k+\Big(([\lambda-\mu_{l_k}]^+)^{1+\gamma}-\sum_{i\neq l_k}^n\mu_i\Big)\Delta_k^{1+\gamma}+
    C\Delta_k^{0.75(1+\gamma)}+(n-1)s_{max}\Delta_k^{\gamma}\mid {\cal{F}}^{\xi}_{\tau_k}\Big]\cr
    &\leq 
    \mathbb{E}\Big[\epsilon\Delta_k-2\epsilon\Delta_k^{1+\gamma}+
    C\Delta_k^{0.75(1+\gamma)}+(n-1)s_{max}\Delta_k^{\gamma}\mid {\cal{F}}^{\xi}_{\tau_k}\Big]=\mathbb{E}[f(\Delta_k)]\mid {\cal{F}}^{\xi}_{\tau_k}],
\end{align*}
where we have used the definitions of $\gamma$ in \eqref{eq:gamma} and $f(\cdot)$ in \eqref{eq:f}. Finally, by \eqref{eq:finally} $\Delta_k>u/s_{max}$, so by \eqref{eq:u} we have $f(\Delta_k)<0$ yielding $\Phi_k \leq 0$. This concludes the proof. 

\qed 

%%%%%%%%%%%%%%%%%%%%%%%%%%%%%%%%%%%%%%%%%%%%%%%%%%%%%
%%%%%%%%%%%%%%%%%%%%%%%%%%%%%%%%%%%%%%%%%%%%%%%%%%%%%
\subsection{Stability of PI-Split} \label{subsec: stability of PI split}

\begin{theorem}\label{main result PI-Split}
Assume $\lambda < \sum_{i=1}^n\mu_i$. Then: \\
(i) $X^S$ is positive recurrent. Consequently, \\
(ii) $X^S$ has a unique stationary distribution, denoted by $\pi^S_{X}$, and \\
(iii) For any initial state $\alpha\in\mathcal{A}$ and any $B \subset \mathcal{A}$, $\mathbb{P}_\alpha(X^S(t)\in B) \rightarrow \pi^S_{X}(B)$ as $t \rightarrow \infty$.
\end{theorem}

\noindent \textit{Proof.} 
The proof is similar to that of Theorem \eqref{main result PI}. Specifically, we prove that the state dependent drift in Lemma \eqref{lem: drift} is also fulfilled by the sampled chain $Y^S$, with the same finite set $A$ \eqref{eq:finite_set}, constants $\epsilon$ \eqref{eq:epsilon}, $\gamma$ \eqref{eq:gamma} and the Lyapunov function ${\cal{L}}$ \eqref{eq:L function}, but with stopping times and a filtration which correspond to the chain $X^S$. Theorem \eqref{main result PI-Split} then follows from the same arguments made in the proof of Theorem \eqref{main result PI}. 

Recall that ${\cal{A}}^*$ denotes the collection of states in which there is at least one queue length equal to zero that does not correspond to the Last Idle server. For initial states $\alpha\in {\cal{A}}^*\setminus A$ in which \textit{there is only one such queue length}, the behavior of PI and PI-Split is identical since batches are not split until (at least) the next sampling time. Thus, the analysis in the proof of Lemma \eqref{lem: drift} proves that the desired state dependent drift holds for these initial states. 

It remains to consider initial states $\alpha\in {\cal{A}}^*\setminus A$ such that there are at least two queue lengths equal to zero which do not correspond to the Last-Idle server. By the definition of the sampling times, the next sampling time must be after exactly 1 time slot. Thus, it remains to prove that
\begin{equation}\label{eq:Split remains to prove}
    \mathbbm{E}_{\alpha}\Big[{\cal{L}}(X^S(1))-{\cal{L}}(X^S(0))\Big]=\sum_{i=1}^n\mathbbm{E}_{\alpha}\Big[(Q_i^S(1))^{1+\gamma}-(Q_i^S(0))^{1+\gamma}\Big]\leq -\epsilon.
\end{equation}
Decompose the sum after the first equality into three parts:\\

\noindent \textbf{Part 1:} the single member of the sum corresponding to the server with the maximum queue length at time slot 0 (if there are several, choose one arbitrarily). Denote the server with the maximal queue length by $j^*=\mbox{argmax}_{1\leq i\leq n}Q^S_i(0)$. By the definition of PI-Split, server $j^*$ does not receive jobs at the beginning of time slot $1$ because its queue length is positive and there are other serves which are idle. By the definition of the finite set in \eqref{eq:finite_set} and \eqref{eq:choice of finite set}, $Q^S_{j^*}(0)\geq s_{max}$. Hence, the queue length at server $j*$ cannot reach zero after one time slot and we have $Q^S_{j^*}(1)=Q^S_{j^*}(0)-s_{j^*}(1)$. \\

\noindent \textbf{Part 2:} the members of the sum corresponding to servers with queue length equal to zero at time slot 0. The queue lengths of these servers at time slot 1 is upper bounded by the number of jobs that arrive at the beginning of time slot 1, namely, $a(1)$. Thus, since $\gamma\leq 1$, $a^2(1)$ is an upper bound for the contribution of each member of part 2 to the sum in \eqref{eq:Split remains to prove}.\\

\noindent \textbf{Part three:} the members of the sum not included in parts 1 and 2. By the definition of PI-Split, members of part 3 do not receive any work. Thus, their contribution to the sum in \eqref{eq:Split remains to prove} is non positive. Specifically, it is upper bounded by $a^2(1)$.\\

We conclude that 
\begin{align}\label{eq:Split leave one}
    \mathbbm{E}_{\alpha}\Big[{\cal{L}}(X^S(1))-{\cal{L}}(X^S(0))\Big]&\leq
    \mathbbm{E}_{\alpha}\Big[(Q_{j^*}^S(1))^{1+\gamma}-(Q_{j^*}^S(0))^{1+\gamma}\Big]+(n-1)\mathbbm{E}_{\alpha}[a^2(1)]\cr
    &=\mathbbm{E}_{\alpha}\Big[(Q_{j^*}^S(0)-s_{j^*}(1))^{1+\gamma}-(Q_{j^*}^S(0))^{1+\gamma}\Big]+n(\sigma_a^2+\lambda^2)\cr
    &\leq (Q_{j^*}^S(0)-1)^{1+\gamma}-(Q_{j^*}^S(0))^{1+\gamma}+n(\sigma_a^2+\lambda^2)\cr
    &\leq -(Q_{j^*}^S(0))^{\gamma}+n(\sigma_a^2+\lambda^2),
\end{align}
where in the second inequality we used the fact that by \eqref{eq:s assumption}, $s_{j^*}(1)\geq 1$ and that given the initial state $\alpha$, $j^*$ and $Q_{j^*}^S(0)$ are known. The last inequality follows from the same calculation as is done in \eqref{eq:taylor}. Finally, by the choice of the finite set in \eqref{eq:choice of finite set} (specifically, the third member in the right hand side of \eqref{eq:choice of finite set}), we obtain
\begin{equation*}
    \mathbbm{E}_{\alpha}\Big[{\cal{L}}(X^S(1))-{\cal{L}}(X^S(0))\Big]\leq -\epsilon,
\end{equation*}
which concludes the proof.
\qed 
%%%%%%%%%%%%%%%%%%%%%%%%%%%%%%%%%%%%%%%%%%%%%%%%%%%%%
% new sim section
%%%%%%%%%%%%%%%%%%%%%%%%%%%%%%%%%%%%%%%%%%%%%%%%%%%%%%%%%%%%%%%%%%%%%%%%%%
%=========================================================================
%  Simulations
%=========================================================================

\section{Simulations}\label{sec:simulations}

%%%%%%%%%%%%%%%%%%%%%%%%%%%%%%%%%%%%%%%%%%%%%%%%%%%%%%%%%%%%%%%%%%%%%%%%%%%%%%%%%%%%%%%%
%%%%%%%%%%%%%%%%%%%%%%%%%%%%%%%%%%%%%%%%%%%%%%%%%%%%%%%%%%%%%%%%%%%%%%%%%%%%%%%%%%%%%%%%
%%%%%%%%%%%%%%%%%%%%%%%%%%%%%%%%%%%%%%%%%%%%%%%%%%%%%%%%%%%%%%%%%%%%%%%%%%%%%%%%%%%%%%%%
%%% transpose
% uplittable matrix
\begin{figure*}

\quad\quad\quad\,\, \underline{\textbf{10 slow : 90 fast}} \quad\quad\quad\quad\quad\quad \underline{\textbf{50 slow : 50 fast}} \quad\quad\quad\quad\quad \underline{\textbf{90 slow : 10 fast}}\par\medskip
%%%%%%%%%%%%%%%%%%%%%%%%%%%%%%%%%%%%%%%%%%%%%%%%%%
\centering
%%%%%%%%%%%%%%%%%%%%%%%%%%%%%%%%%%%%%%%%%%%%%%%%%%
\begin{subfigure}
\centering
\begin{subfigure}{}
\includegraphics[width=0.32\linewidth]{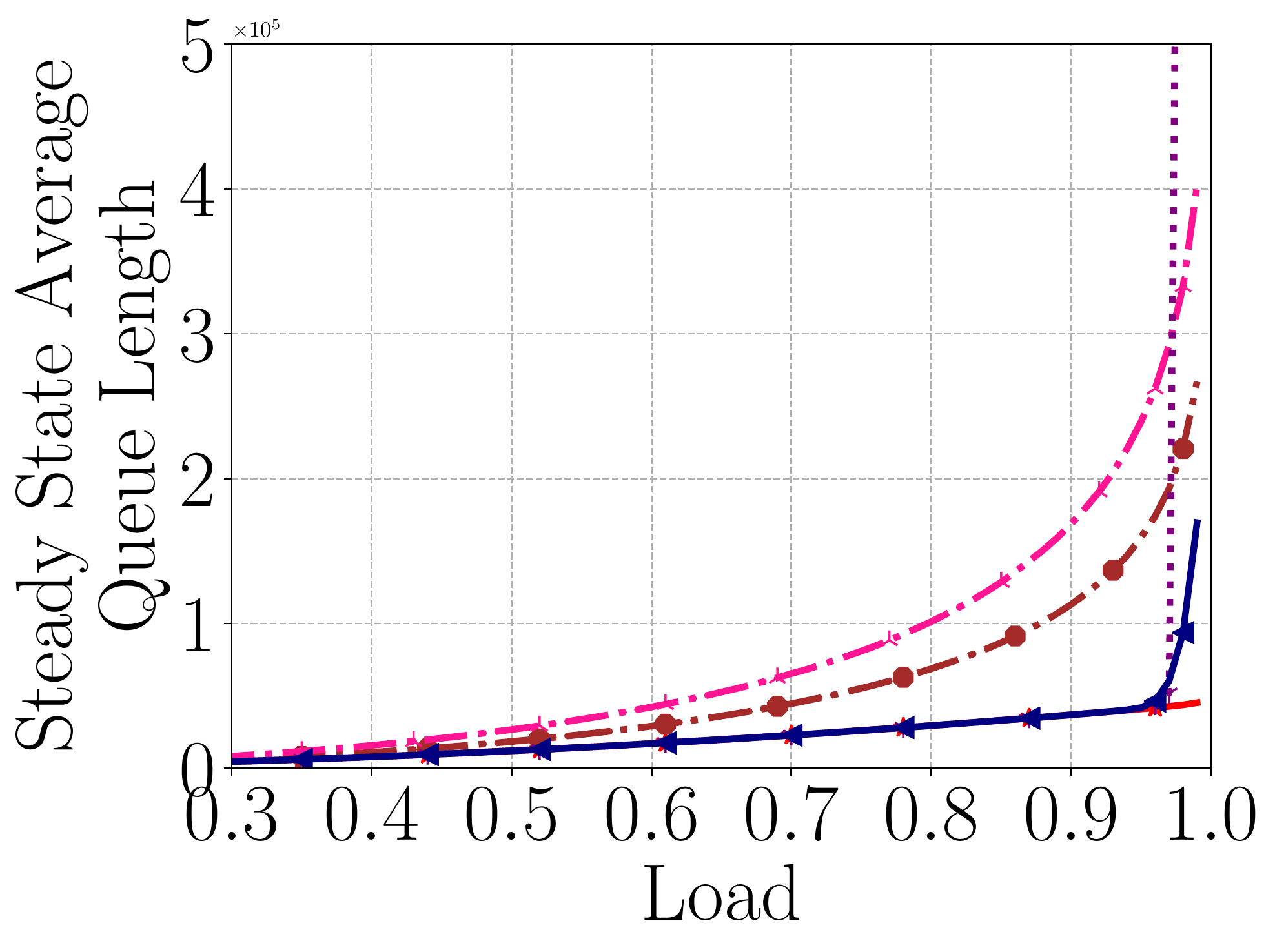}
\end{subfigure}%
\begin{subfigure}{}
\includegraphics[width=0.32\linewidth]{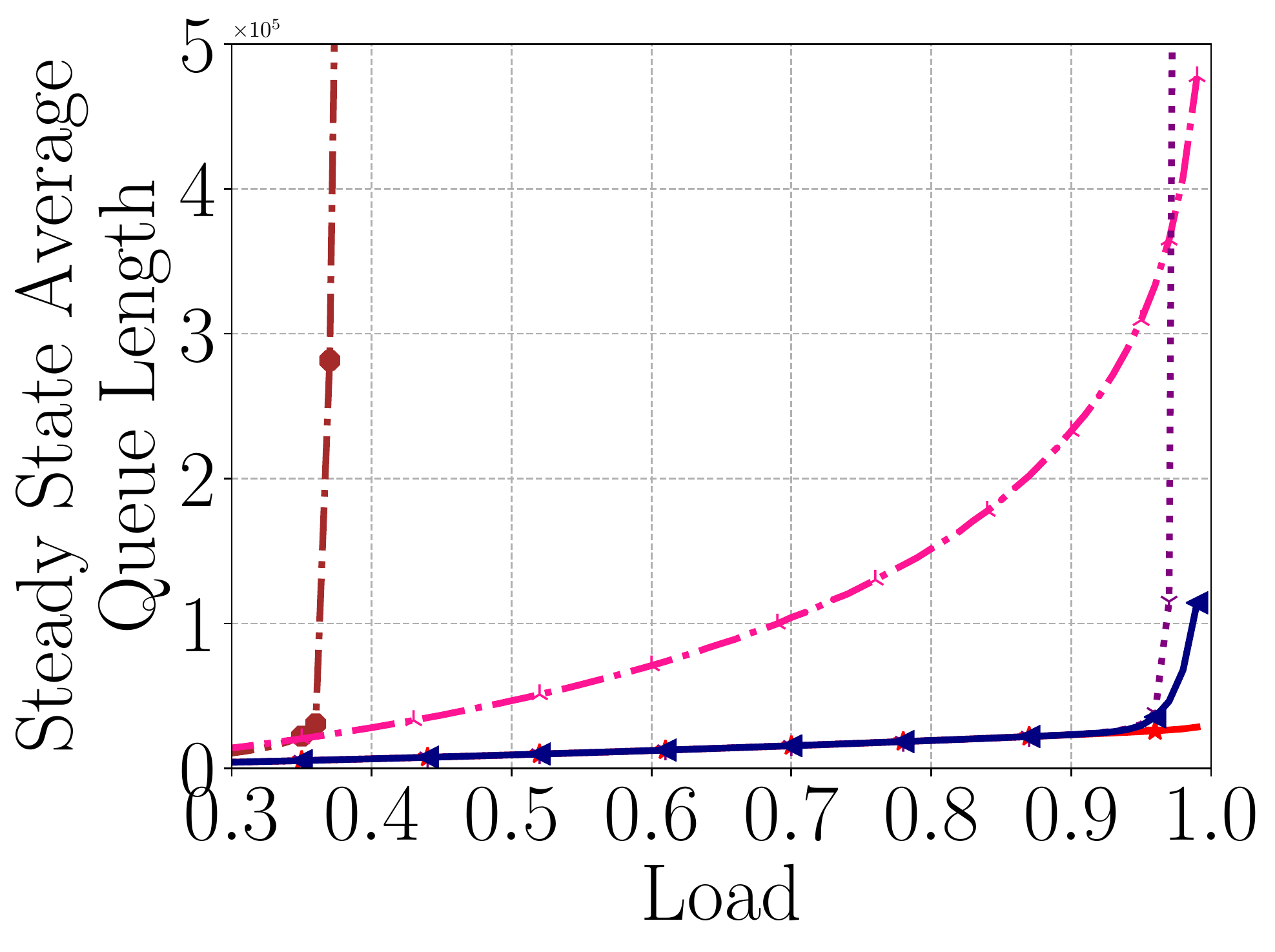}
\end{subfigure}%
\begin{subfigure}{}
\includegraphics[width=0.32\linewidth]{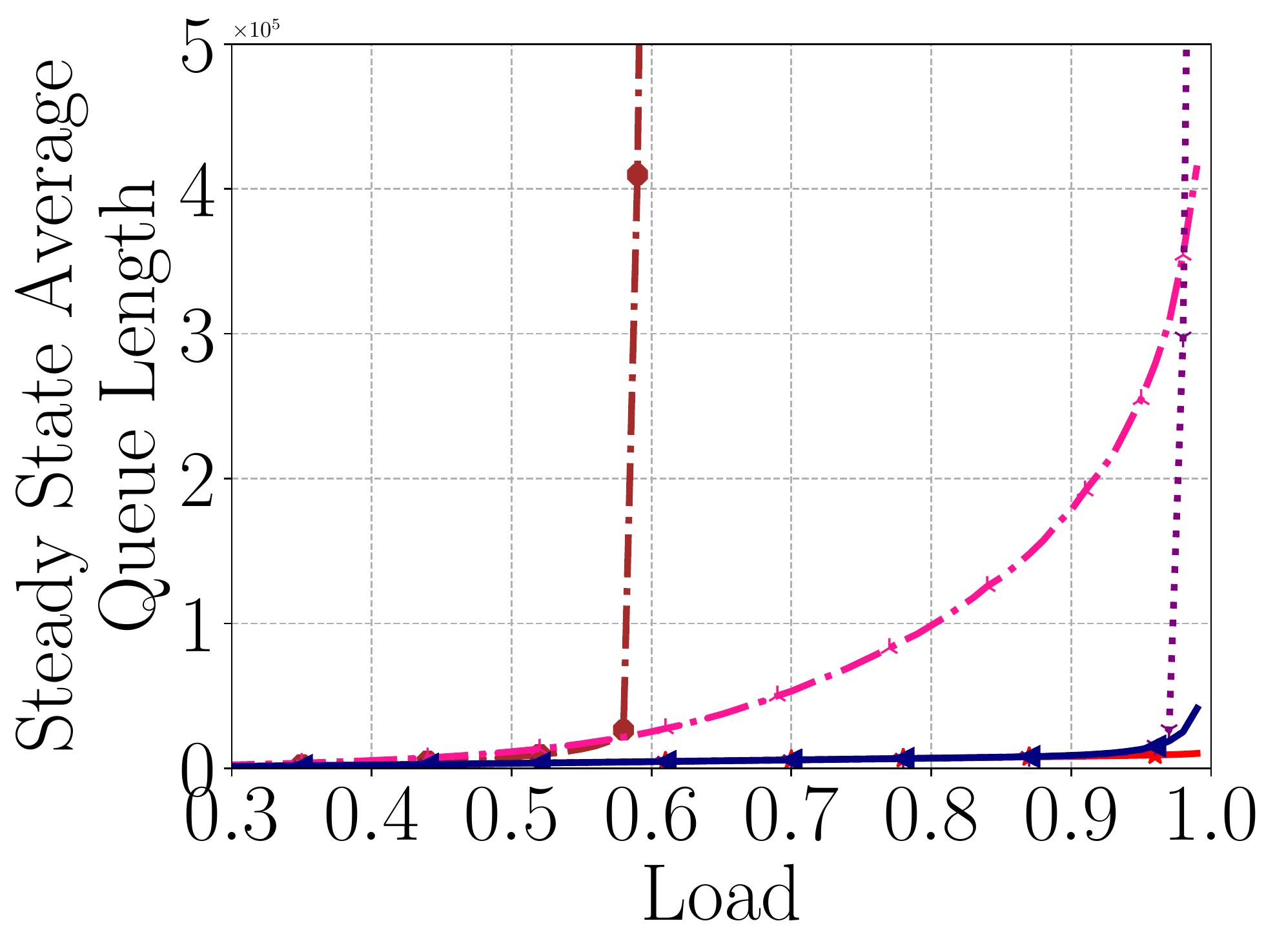}
\end{subfigure}%
\end{subfigure}
%%%%%%%%%%%%%%%%%%%%%%%%%%%%%%%%%%%%%%%%%%%%%%%%%%
\begin{subfigure}
\centering
\begin{subfigure}{}
\includegraphics[width=0.32\linewidth]{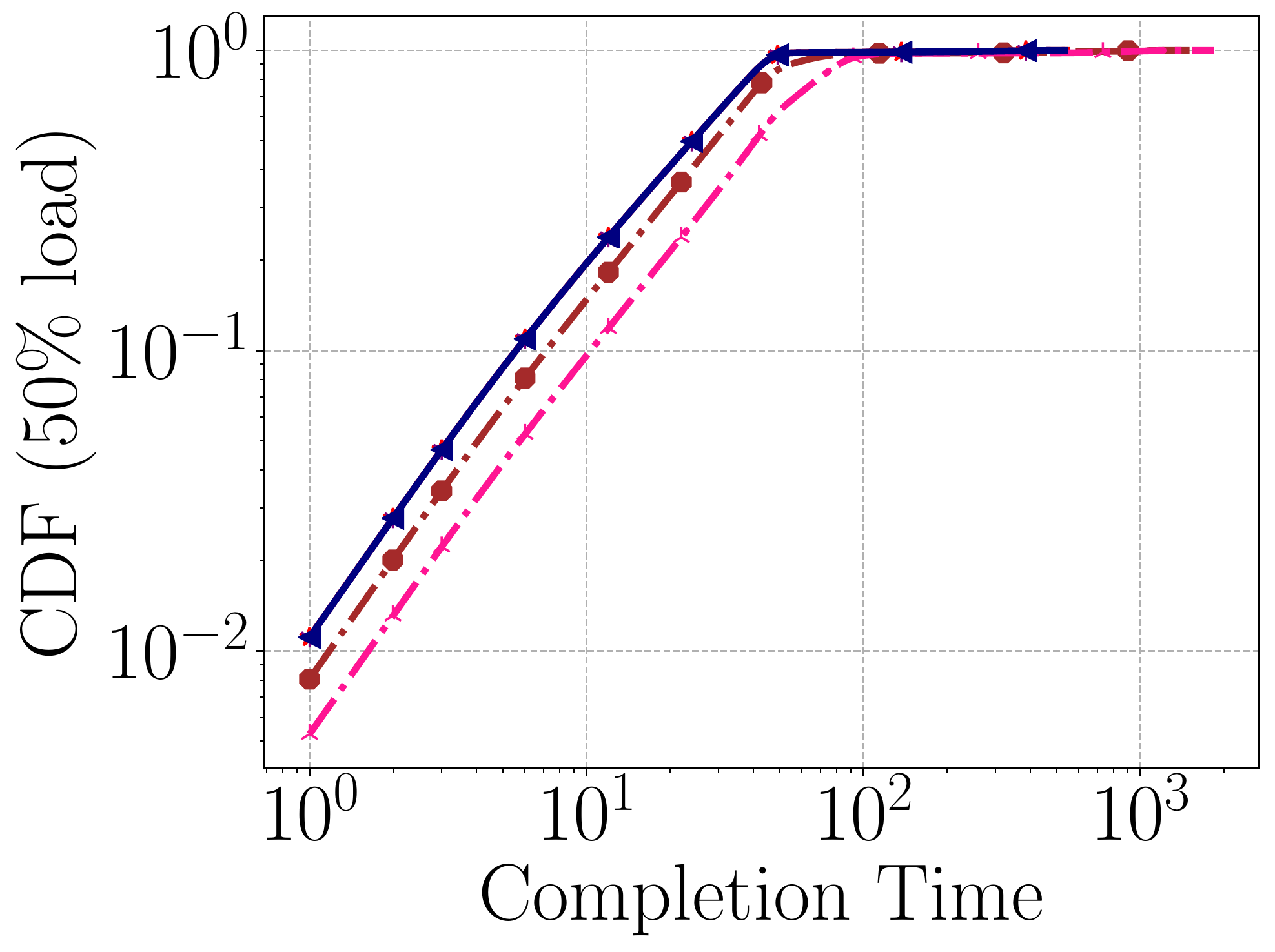}
\end{subfigure}%
\begin{subfigure}{}
\includegraphics[width=0.32\linewidth]{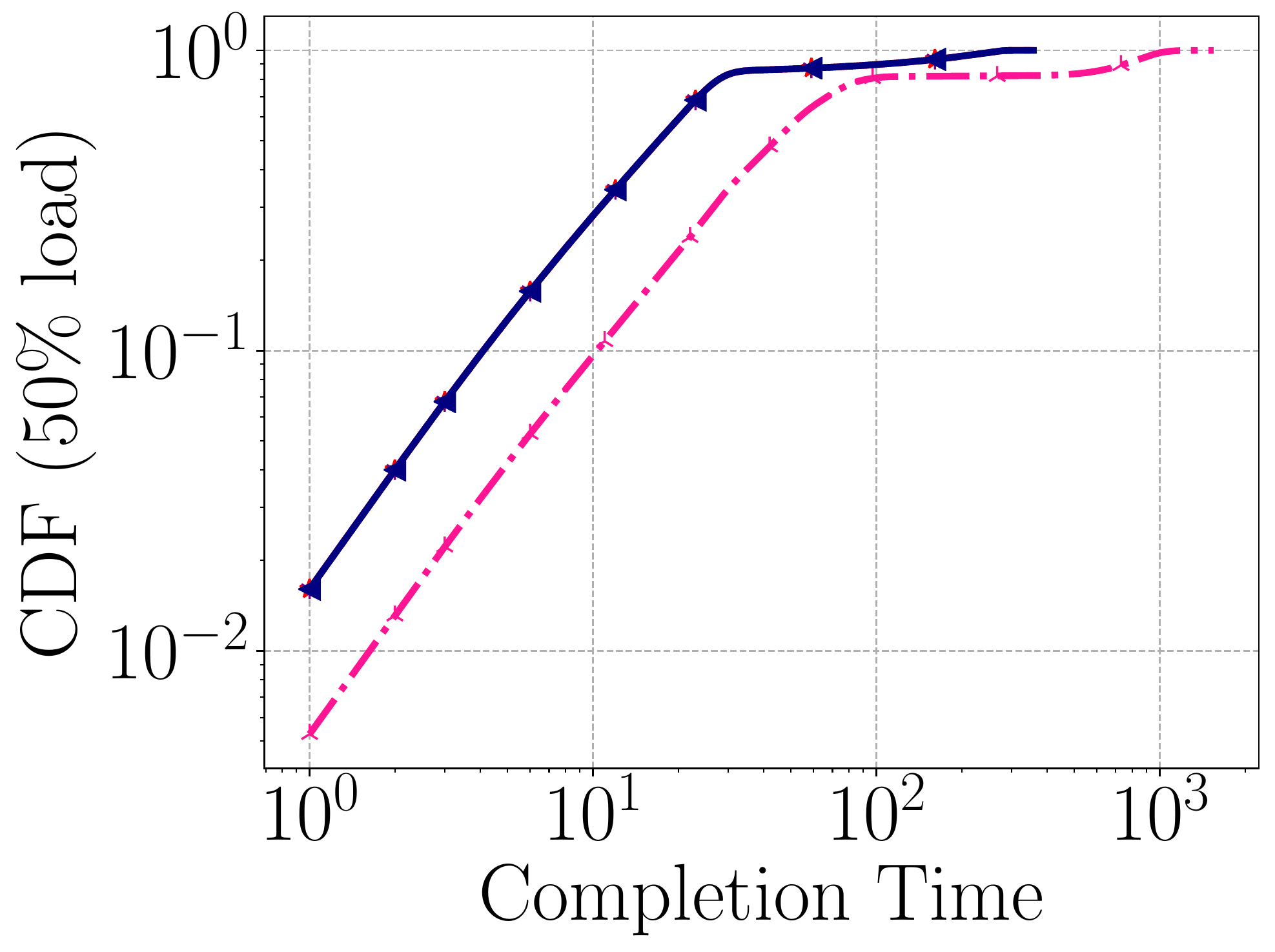}
\end{subfigure}%
\begin{subfigure}{}
\includegraphics[width=0.32\linewidth]{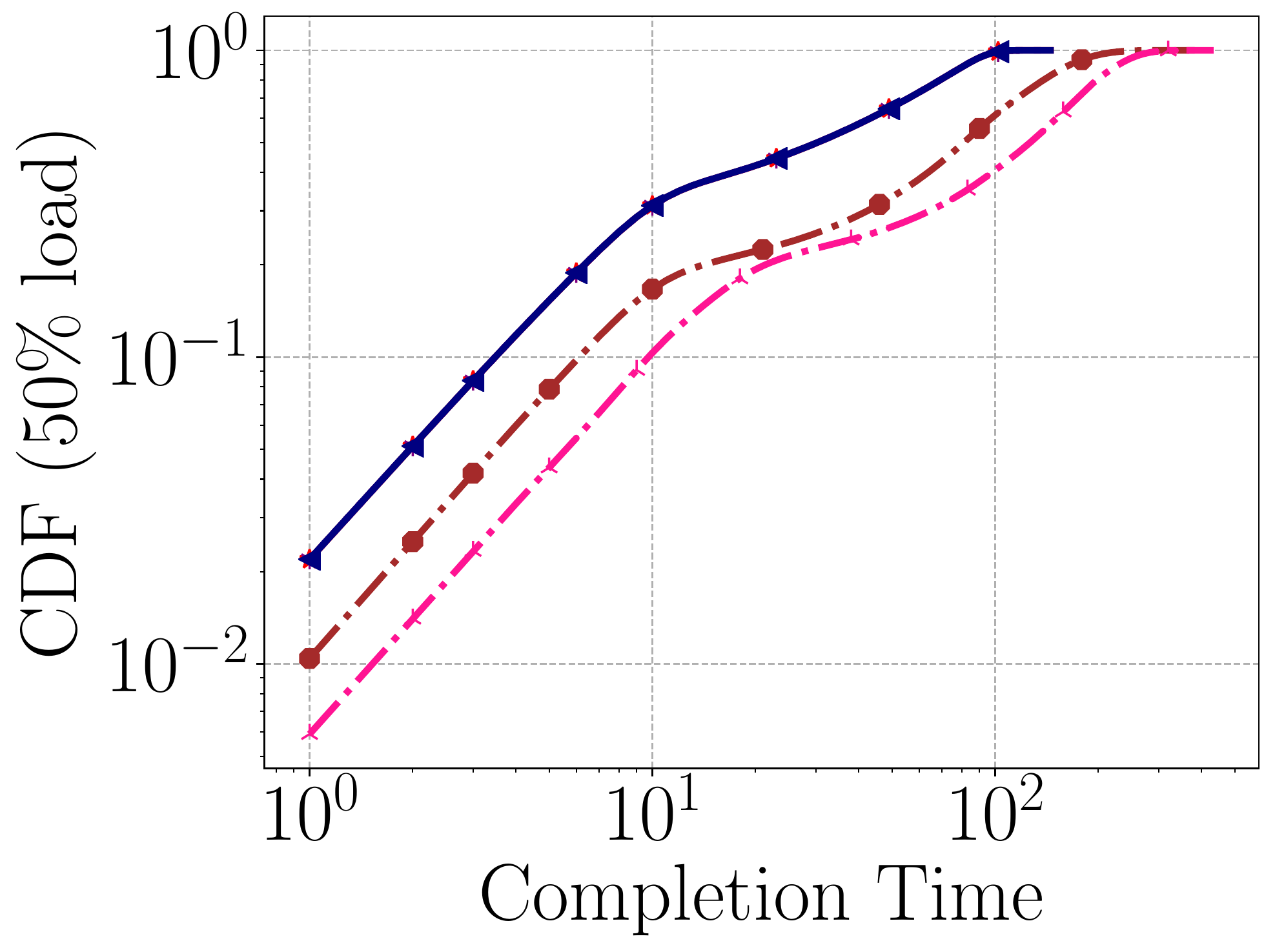}
\end{subfigure}%
\end{subfigure}
%%%%%%%%%%%%%%%%%%%%%%%%%%%%%%%%%%%%%%%%%%%%%%%%%%
\begin{subfigure}
\centering
\begin{subfigure}{}
\includegraphics[width=0.32\linewidth]{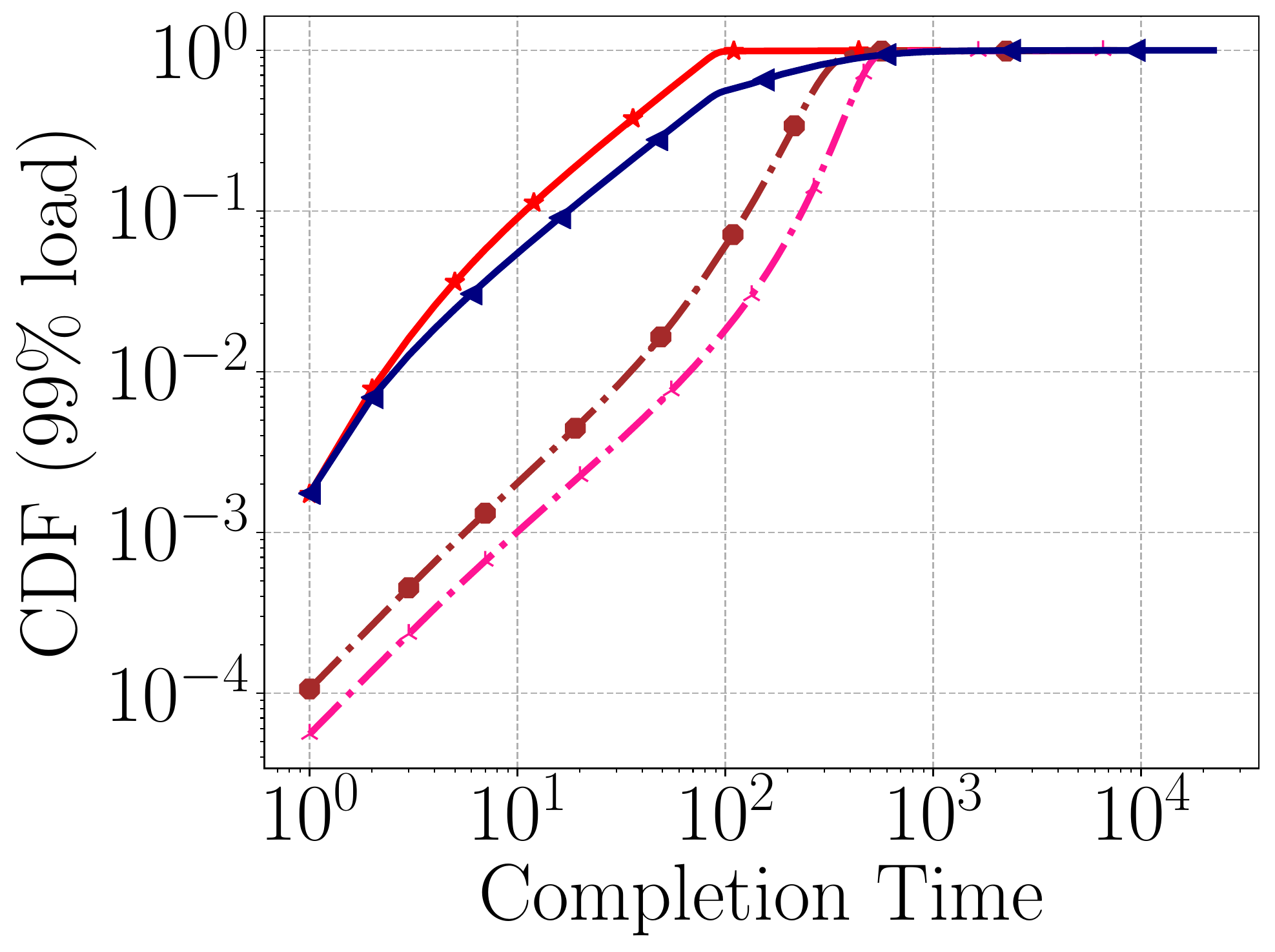}
\end{subfigure}%
\begin{subfigure}{}
\includegraphics[width=0.32\linewidth]{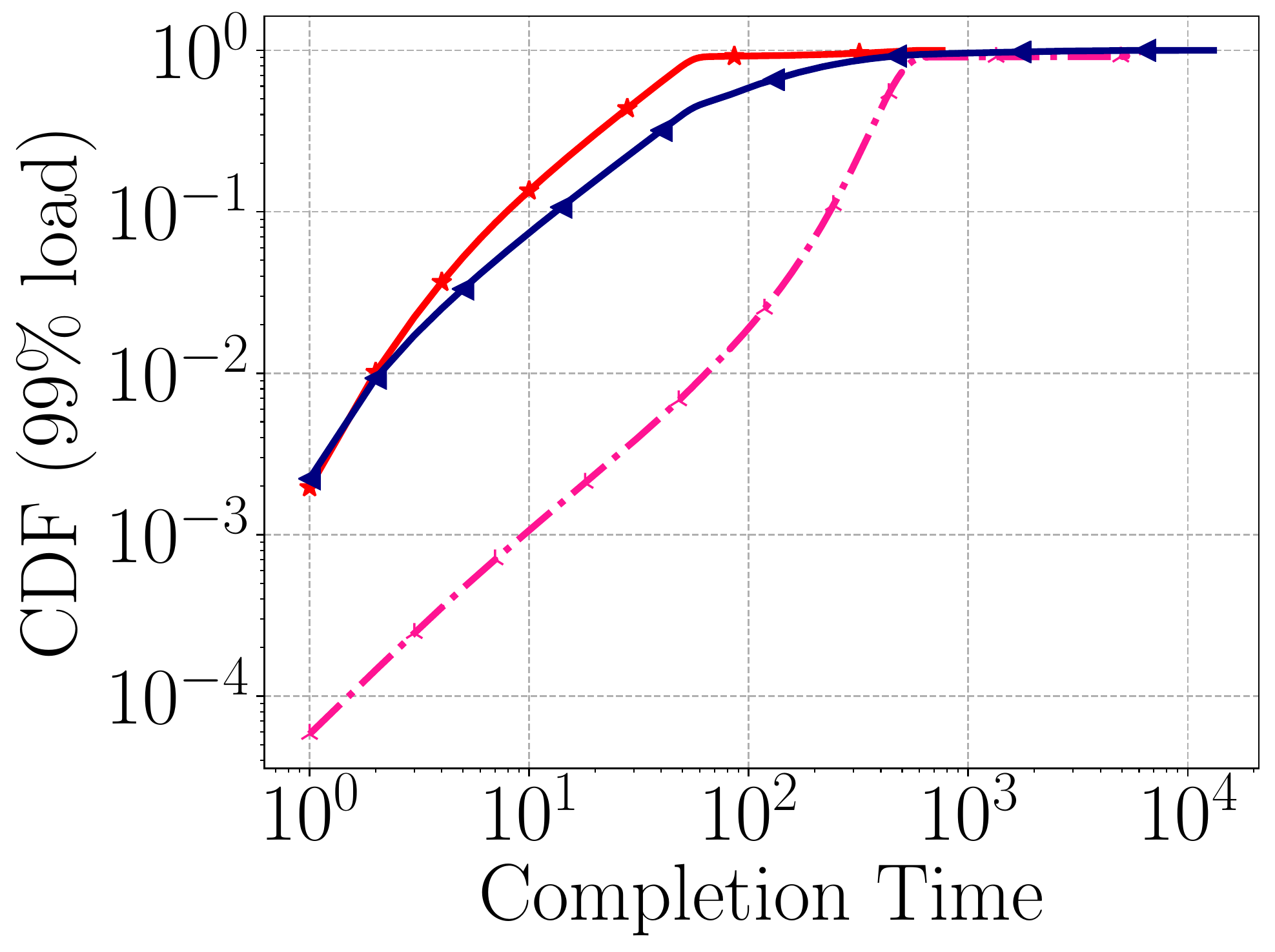}
\end{subfigure}%
\begin{subfigure}{}
\includegraphics[width=0.32\linewidth]{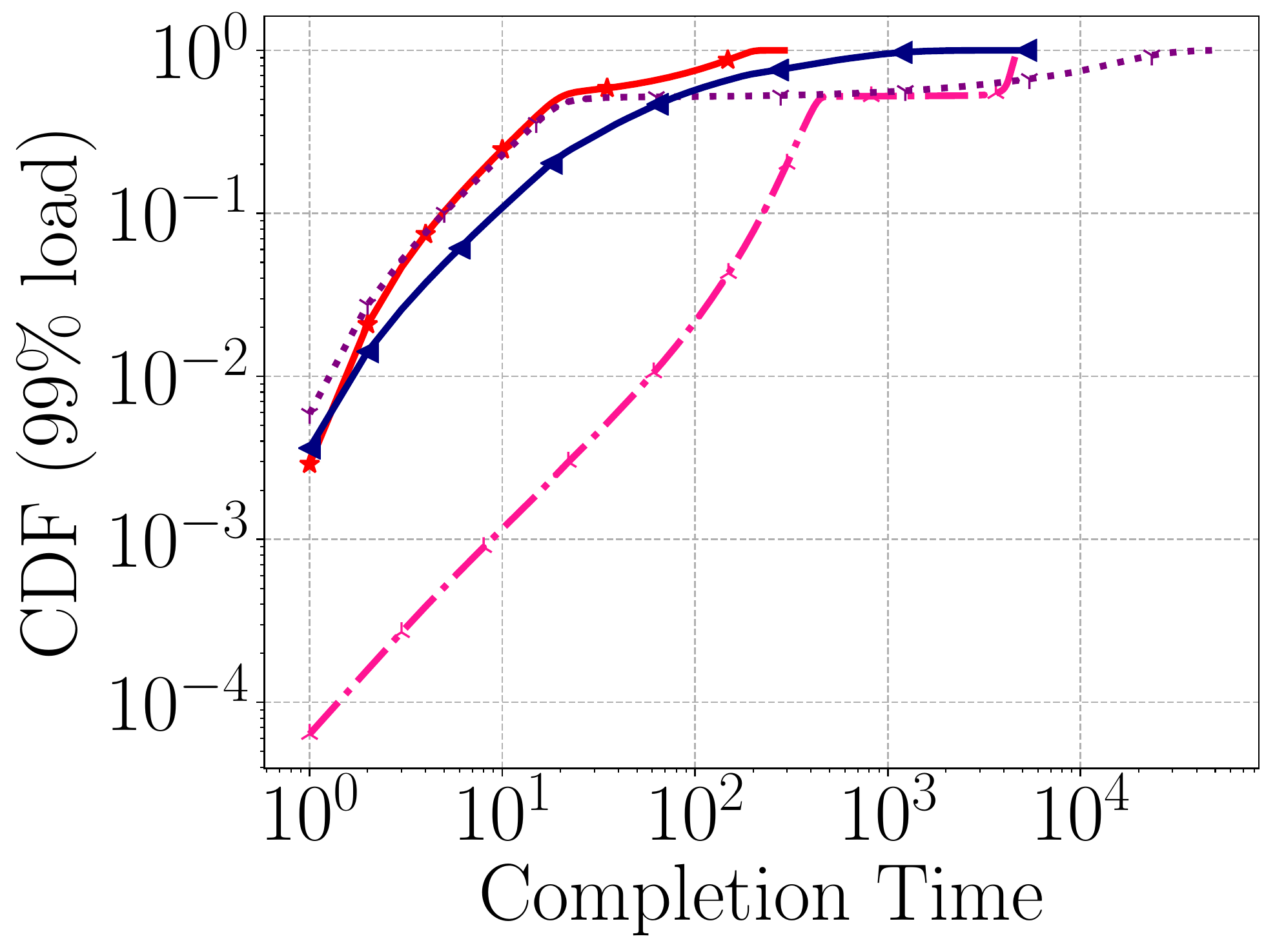}
\end{subfigure}%
\end{subfigure}
%%%%%%%%%%%%%%%%%%%%%%%%%%%%%%%%%%%%%%%%%%%%%%%%%%
\begin{subfigure}
\centering
\begin{subfigure}{}
\includegraphics[width=0.32\linewidth]{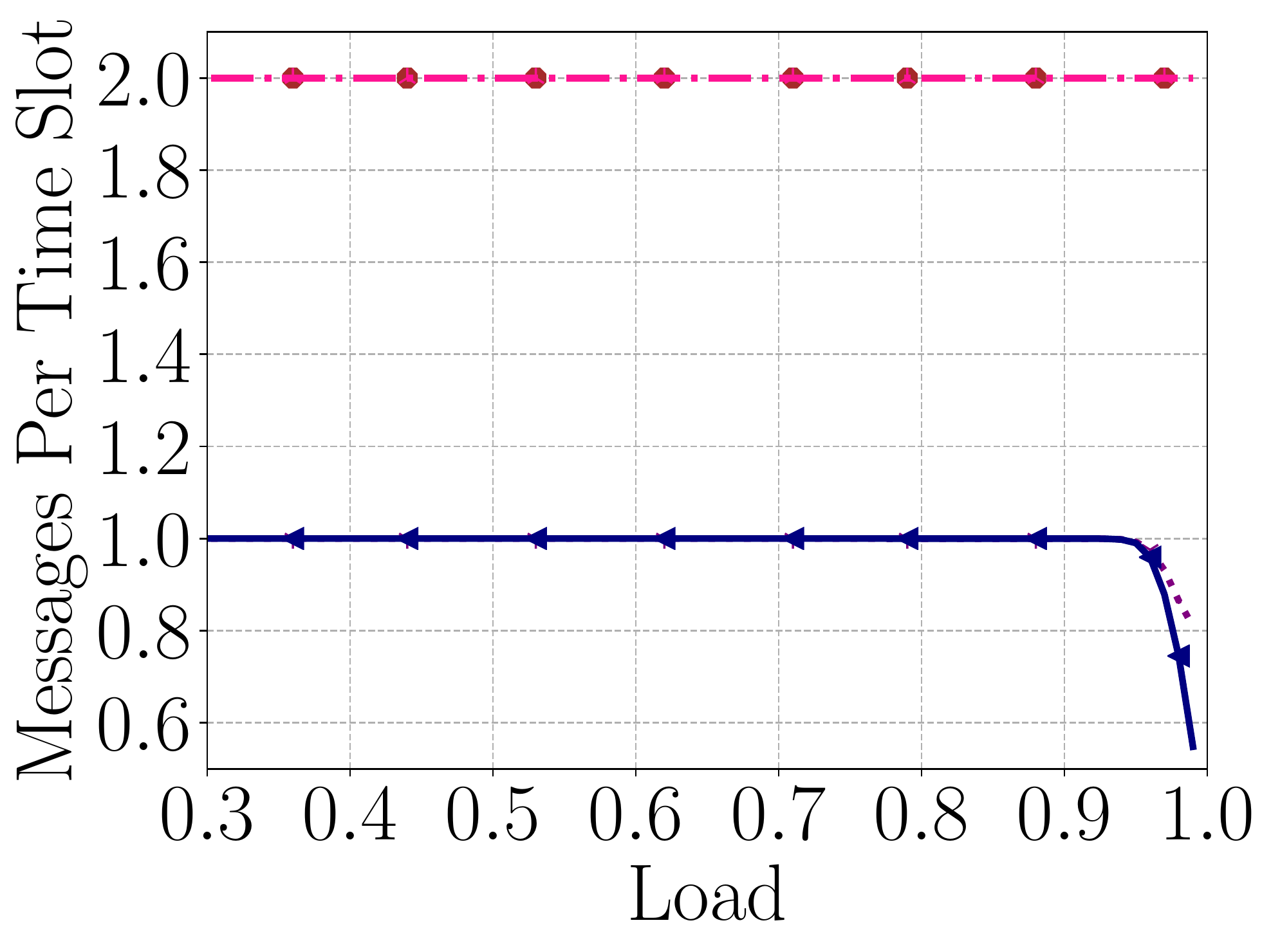}
\end{subfigure}%
\begin{subfigure}{}
\includegraphics[width=0.32\linewidth]{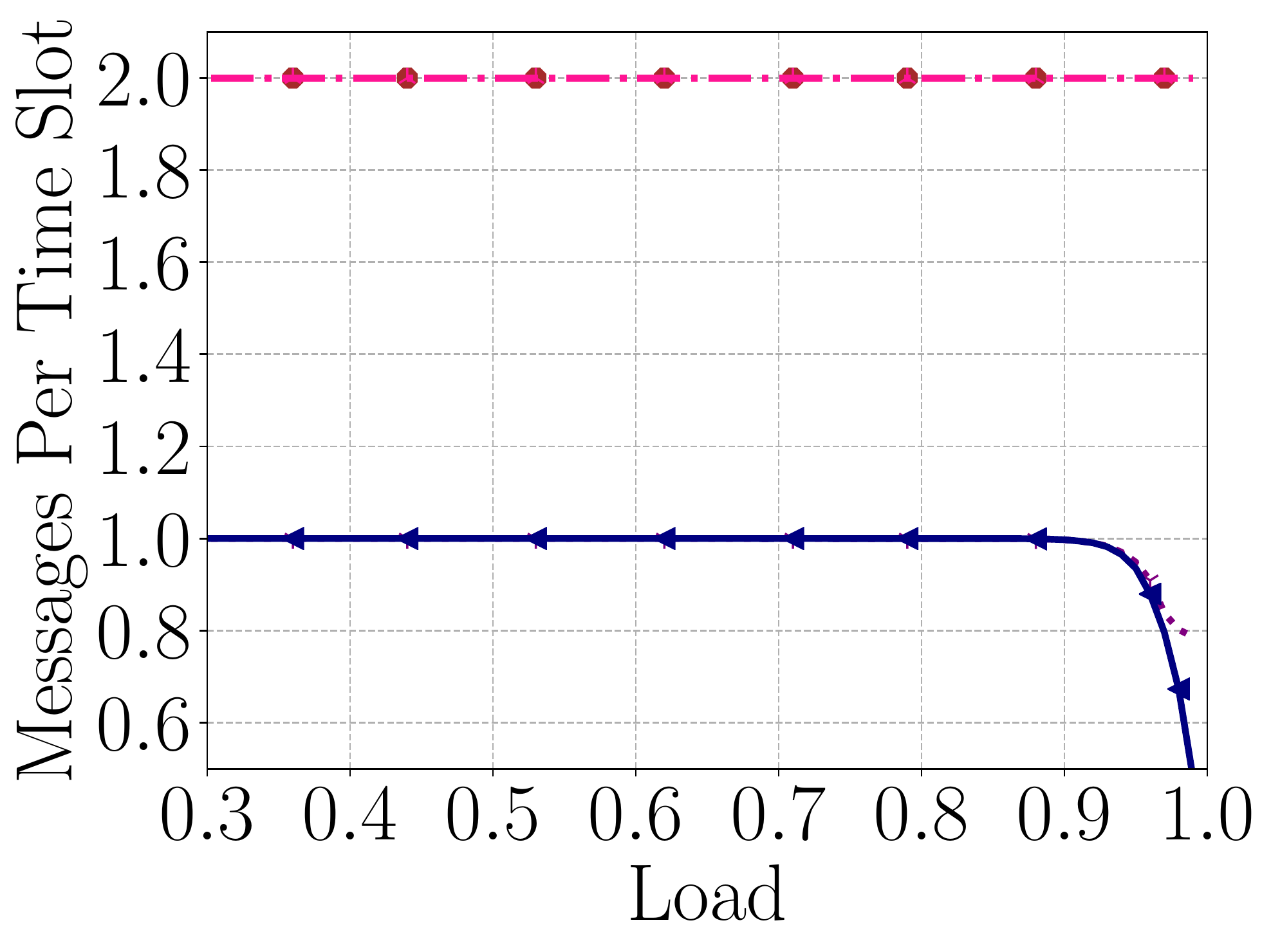}
\end{subfigure}%
\begin{subfigure}{}
\includegraphics[width=0.32\linewidth]{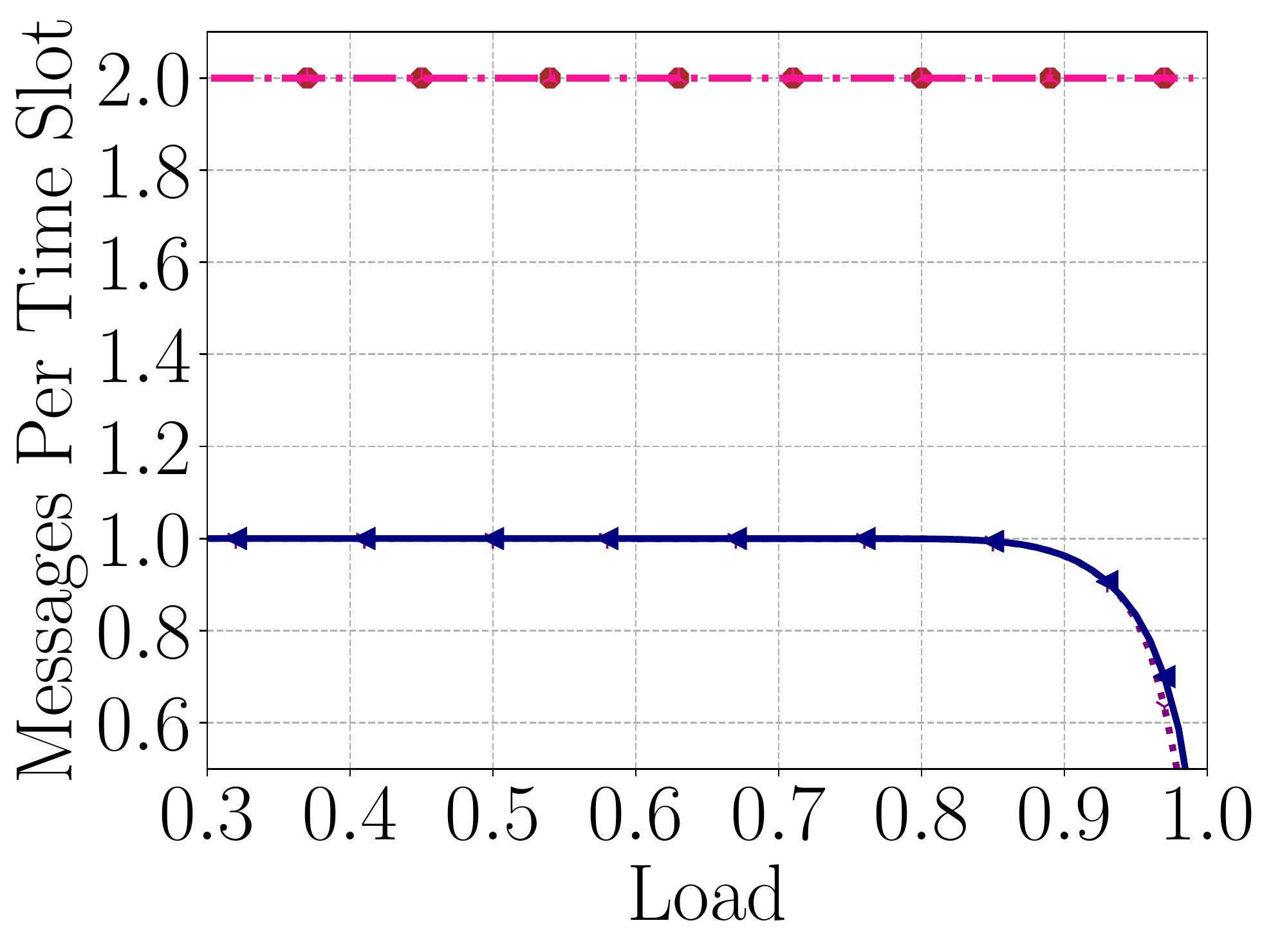}
\end{subfigure}%
\end{subfigure}
%%%%%%%%%%%%%%%%%%%%%%%%%%%%%%%%%%%%%%%%%%%%%%%%%%
\begin{subfigure}{}
\includegraphics[width=0.9\linewidth]{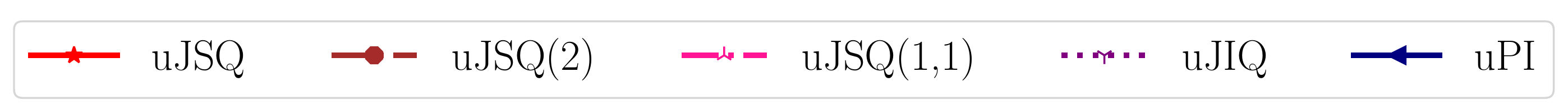}
\end{subfigure}%
%%%%%%%%%%%%%%%%%%%%%%%%%%%%%%%%%%%%%%%%%%%%%%%%%%
\caption{Performance of different policies with 100 heterogeneous servers and different slow:fast server ratios. The policies we consider here cannot split arriving batches of jobs. We emphasize this by adding 'u' before each policy name (which stands for 'unsplittable').}
\label{fig:eval:unsplittable}
\end{figure*}

\begin{figure*}

\quad\quad\quad\,\, \underline{\textbf{10 slow : 90 fast}} \quad\quad\quad\quad\quad\quad \underline{\textbf{50 slow : 50 fast}} \quad\quad\quad\quad\quad \underline{\textbf{90 slow : 10 fast}}\par\medskip
%%%%%%%%%%%%%%%%%%%%%%%%%%%%%%%%%%%%%%%%%%%%%%%%%%
\centering
%%%%%%%%%%%%%%%%%%%%%%%%%%%%%%%%%%%%%%%%%%%%%%%%%%
\begin{subfigure}
\centering
\begin{subfigure}{}
\includegraphics[width=0.32\linewidth]{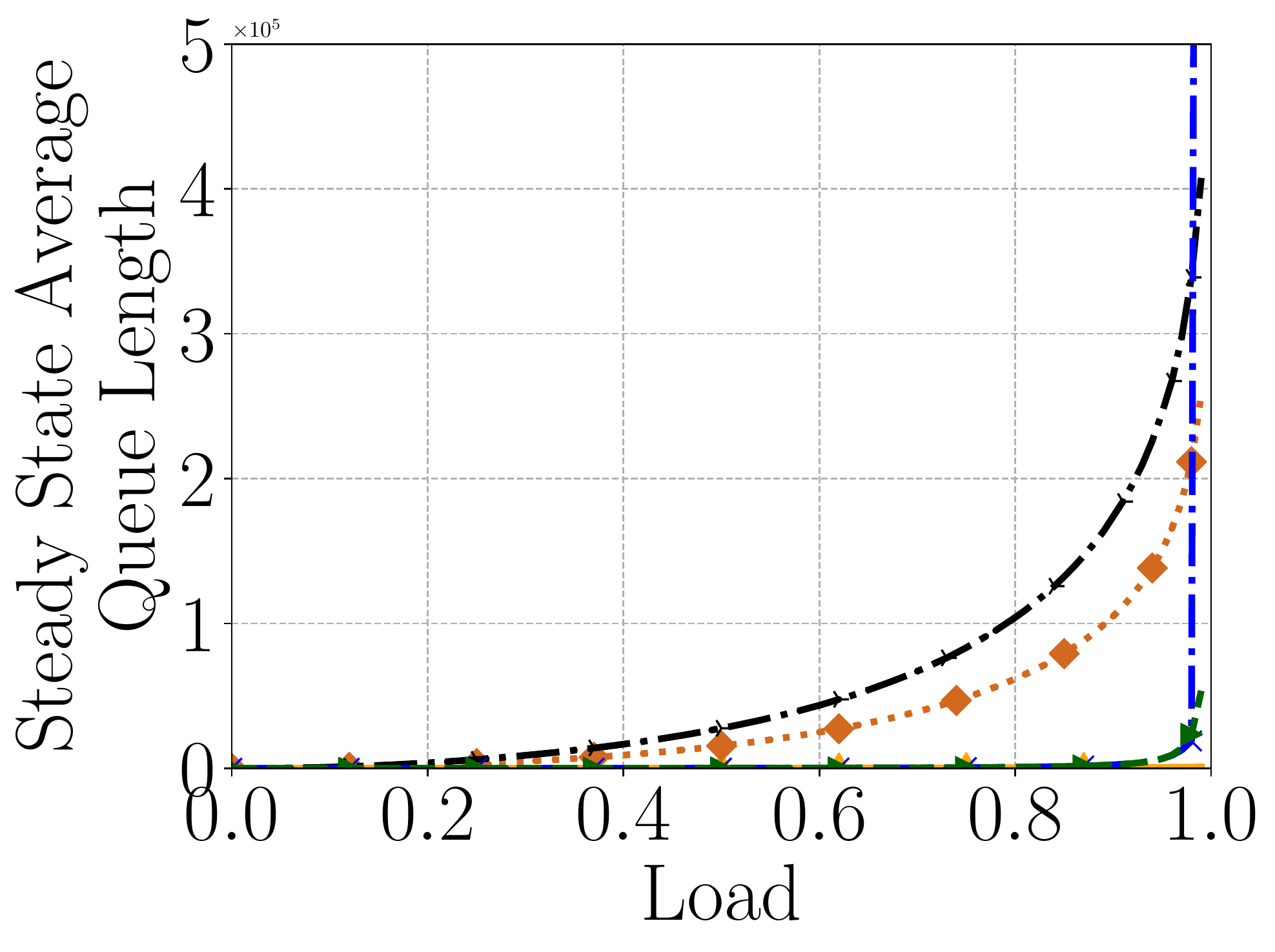}
\end{subfigure}%
\begin{subfigure}{}
\includegraphics[width=0.32\linewidth]{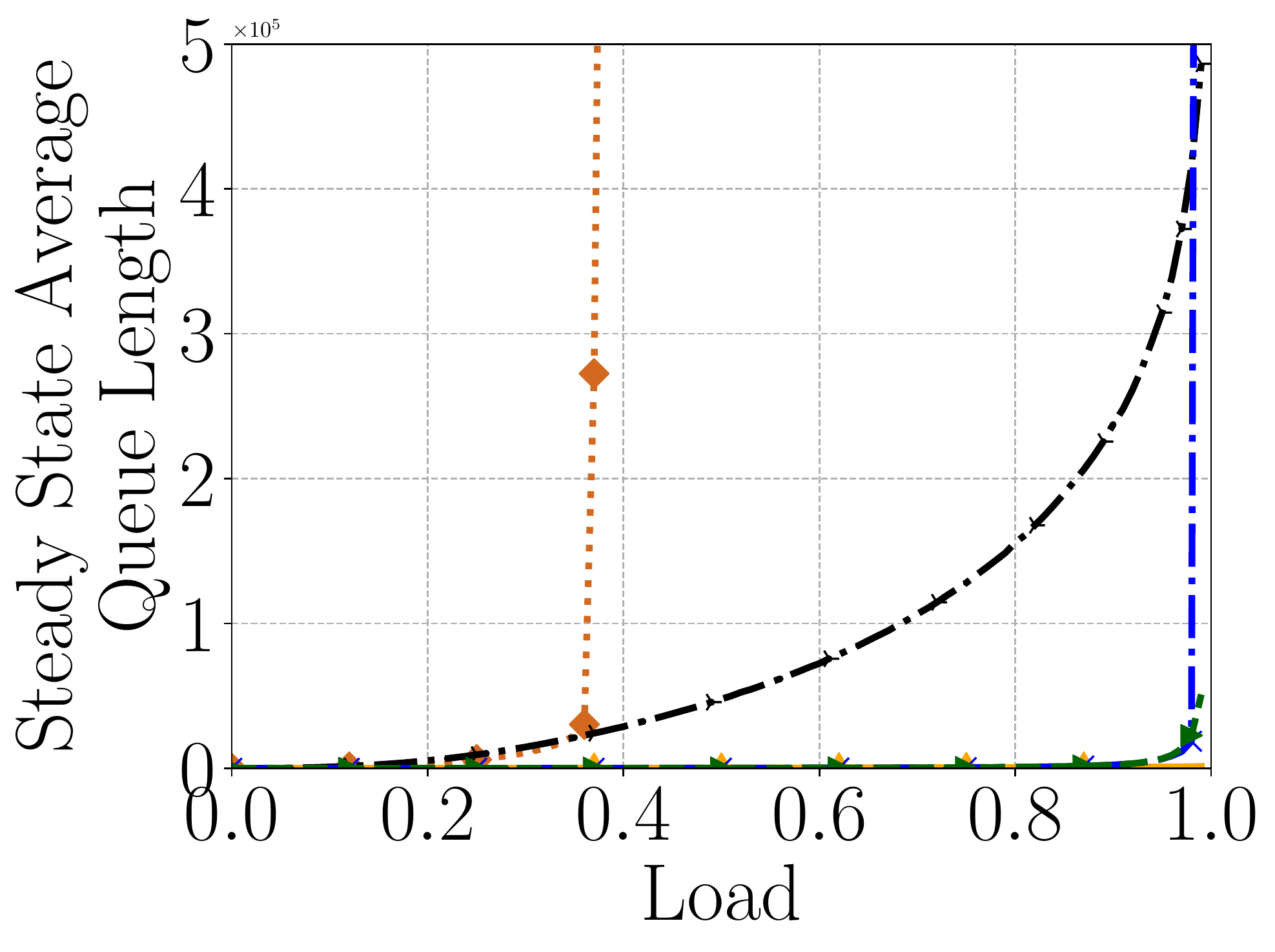}
\end{subfigure}%
\begin{subfigure}{}
\includegraphics[width=0.32\linewidth]{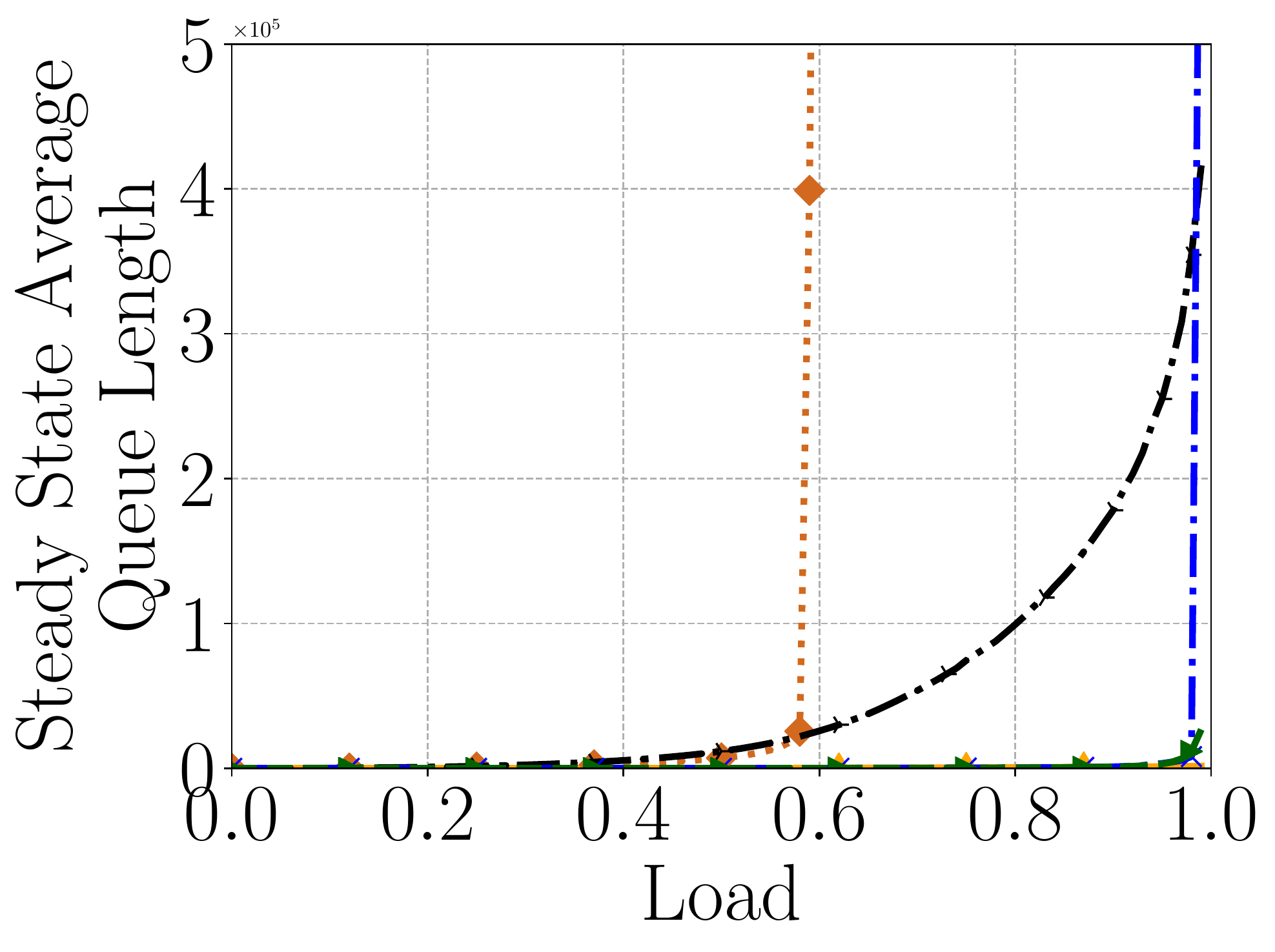}
\end{subfigure}%
\end{subfigure}
%%%%%%%%%%%%%%%%%%%%%%%%%%%%%%%%%%%%%%%%%%%%%%%%%%
\begin{subfigure}
\centering
\begin{subfigure}{}
\includegraphics[width=0.32\linewidth]{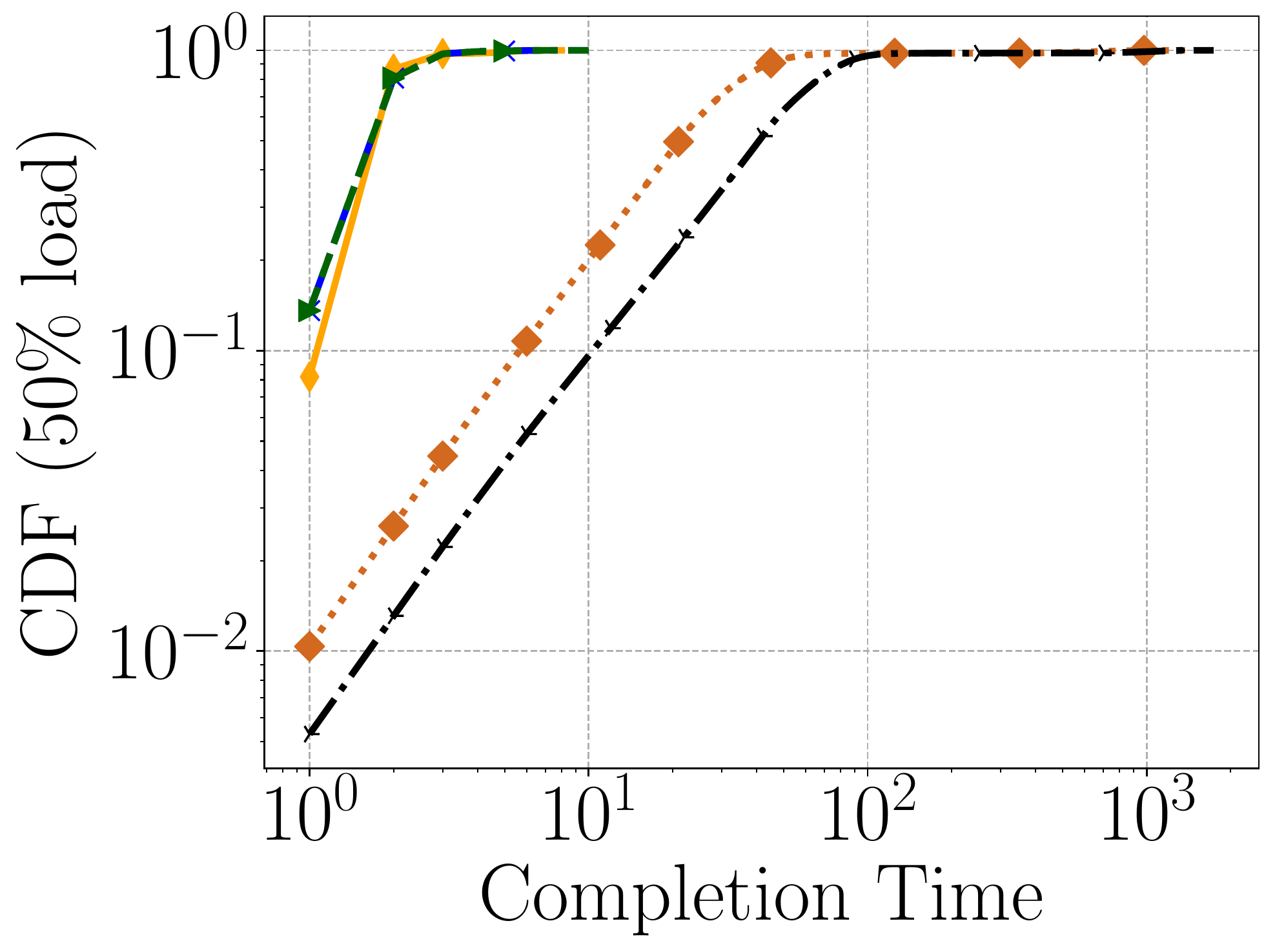}
\end{subfigure}%
\begin{subfigure}{}
\includegraphics[width=0.32\linewidth]{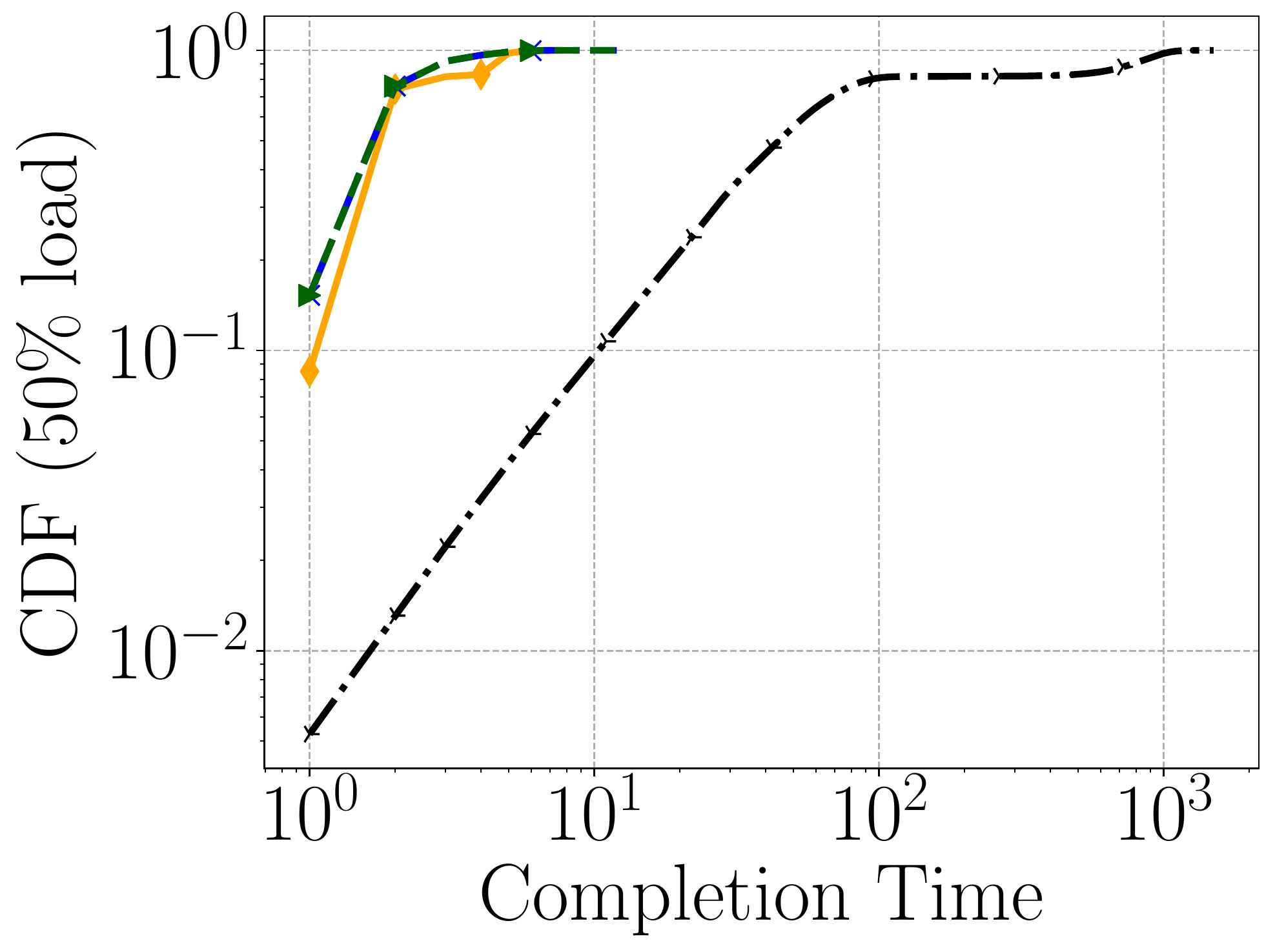}
\end{subfigure}%
\begin{subfigure}{}
\includegraphics[width=0.32\linewidth]{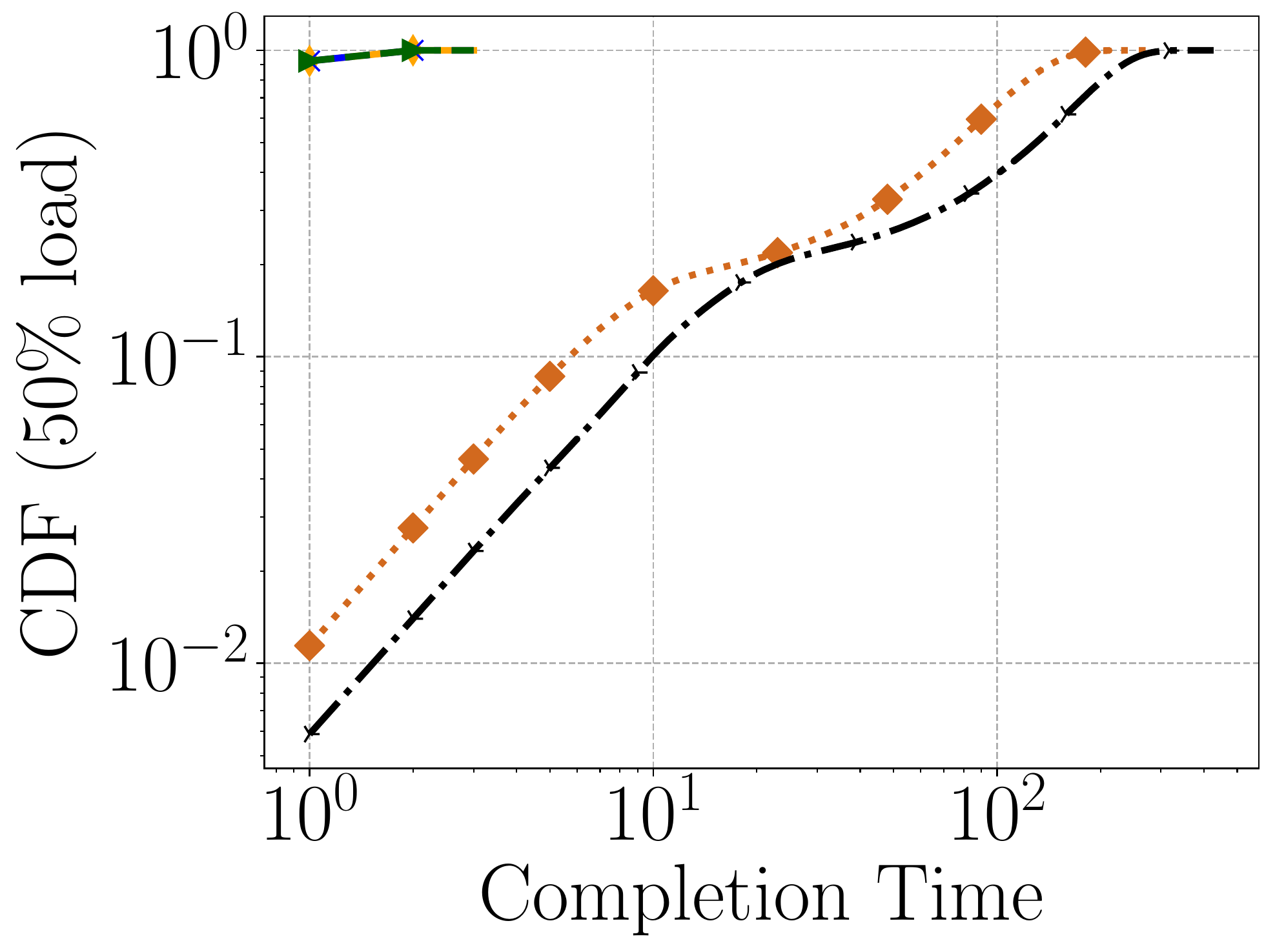}
\end{subfigure}%
\end{subfigure}
%%%%%%%%%%%%%%%%%%%%%%%%%%%%%%%%%%%%%%%%%%%%%%%%%%
\begin{subfigure}
\centering
\begin{subfigure}{}
\includegraphics[width=0.32\linewidth]{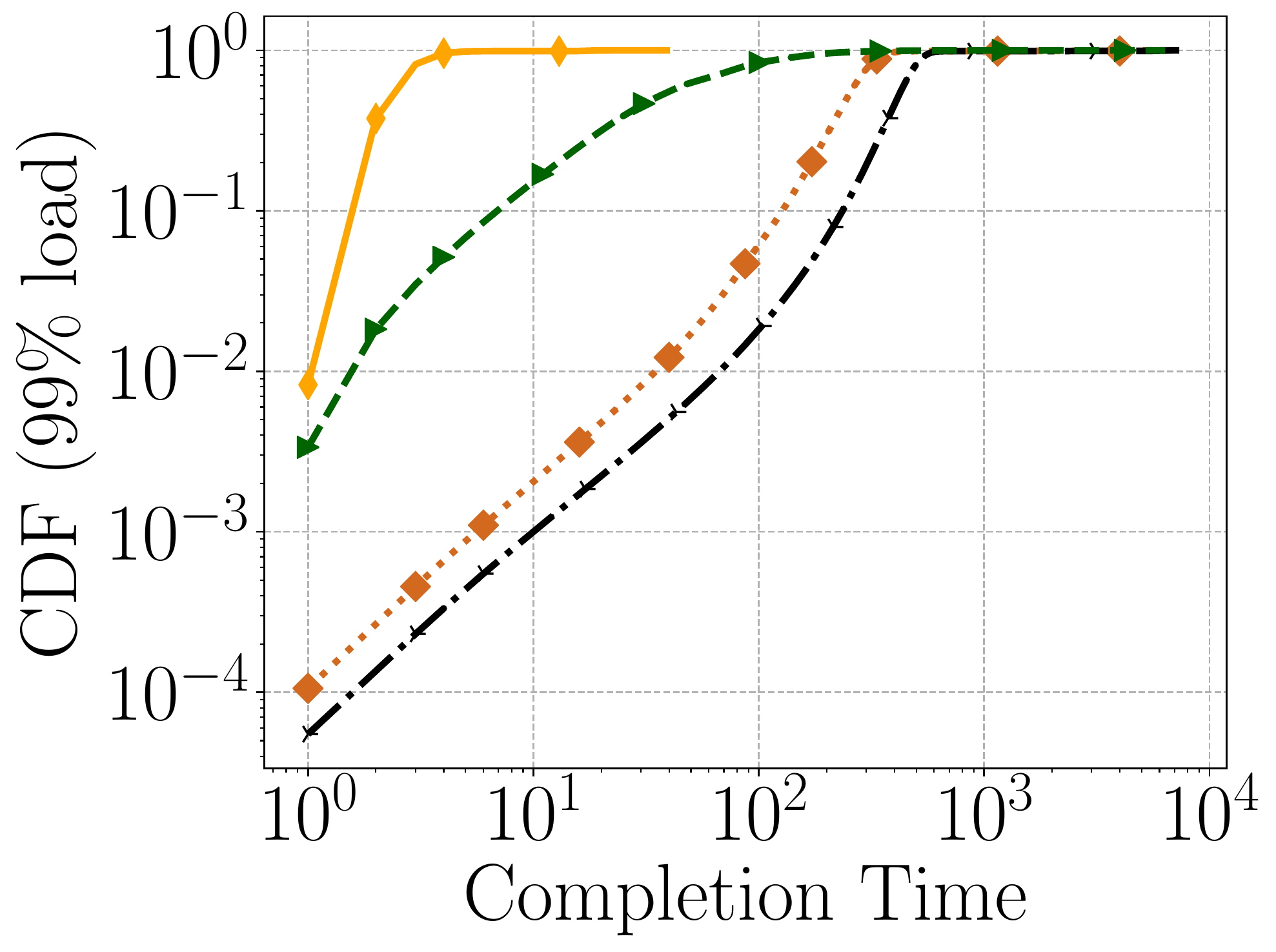}
\end{subfigure}%
\begin{subfigure}{}
\includegraphics[width=0.32\linewidth]{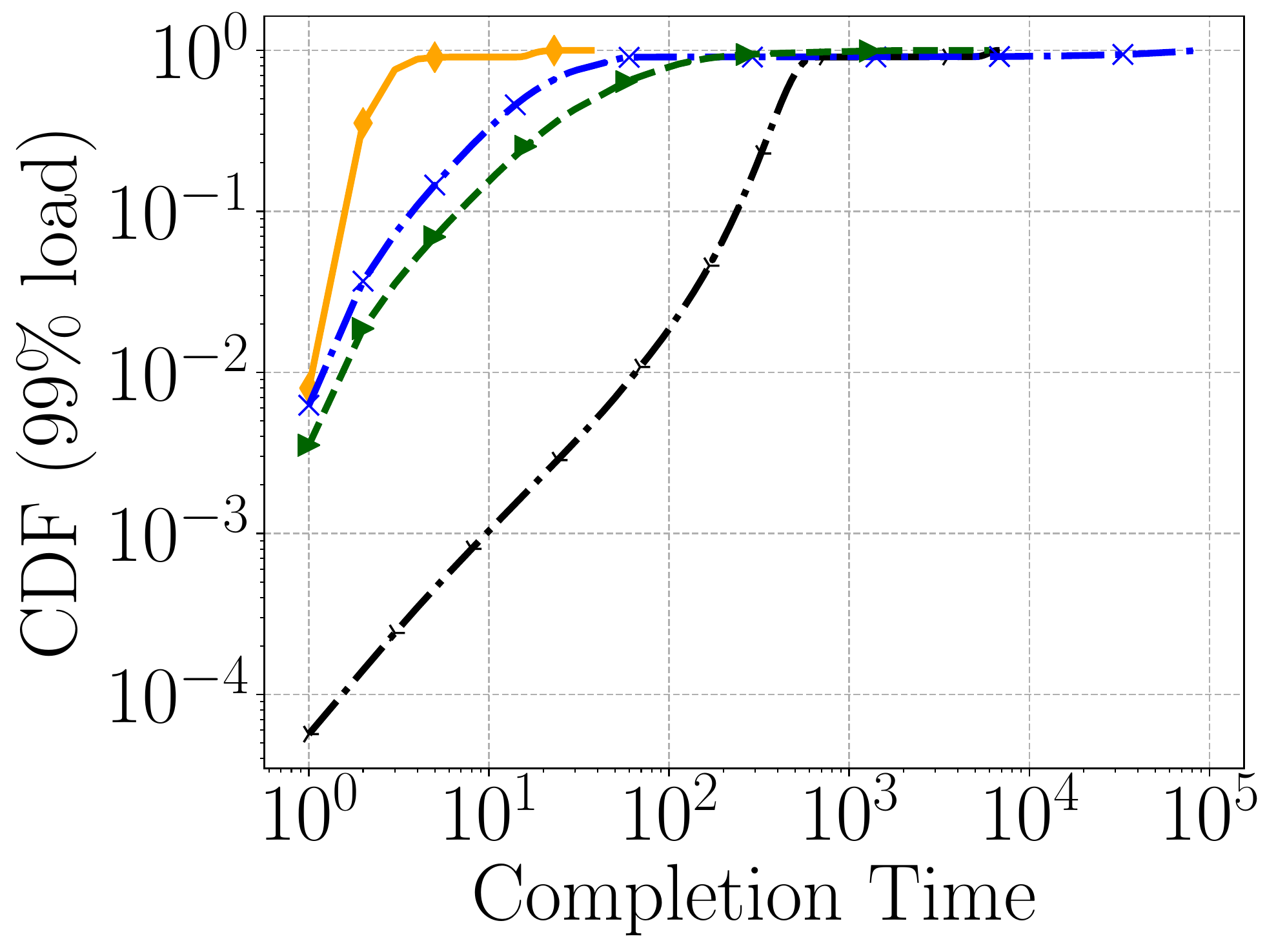}
\end{subfigure}%
\begin{subfigure}{}
\includegraphics[width=0.32\linewidth]{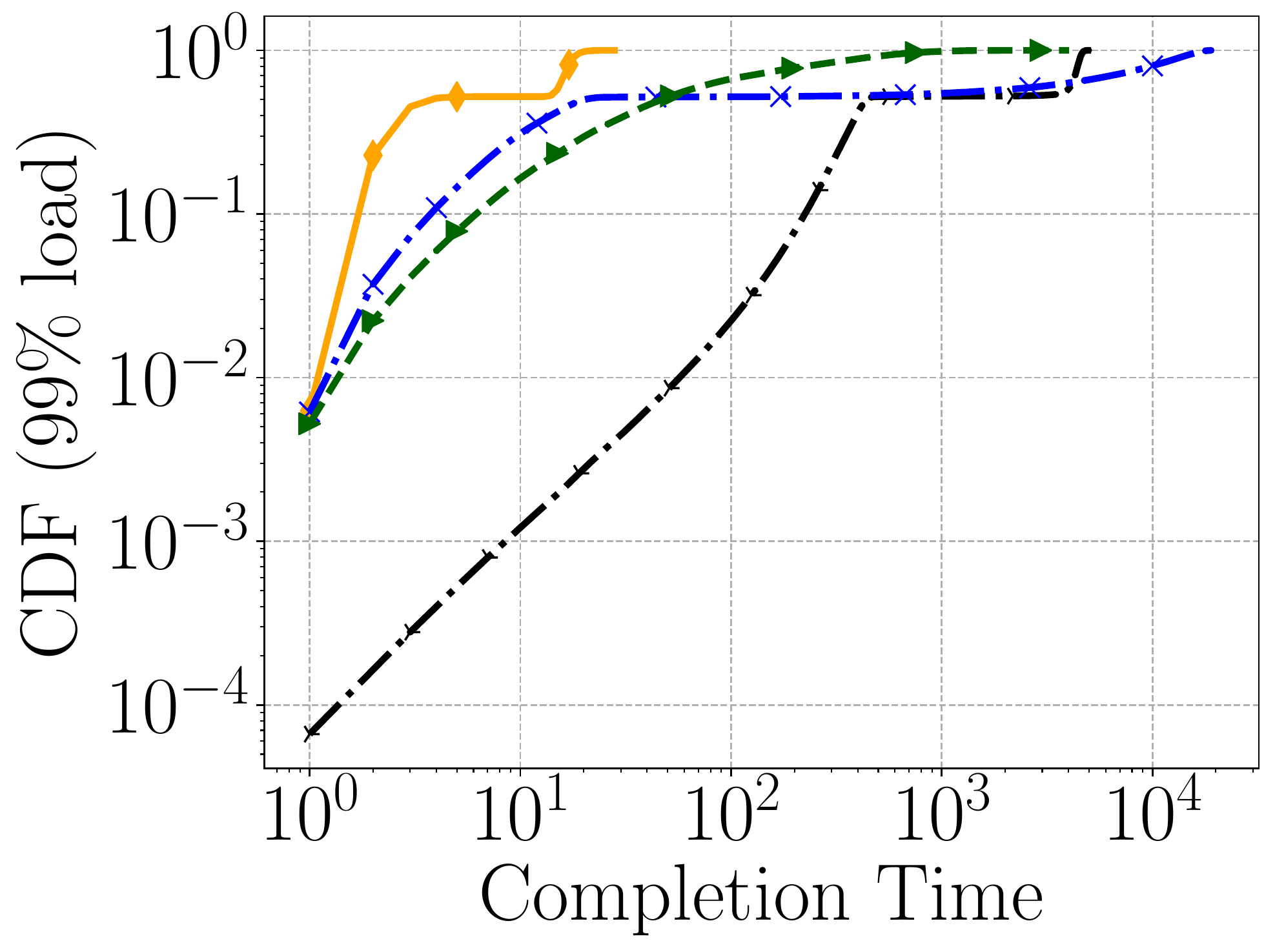}
\end{subfigure}%
\end{subfigure}
%%%%%%%%%%%%%%%%%%%%%%%%%%%%%%%%%%%%%%%%%%%%%%%%%%
\begin{subfigure}
\centering
\begin{subfigure}{}
\includegraphics[width=0.32\linewidth]{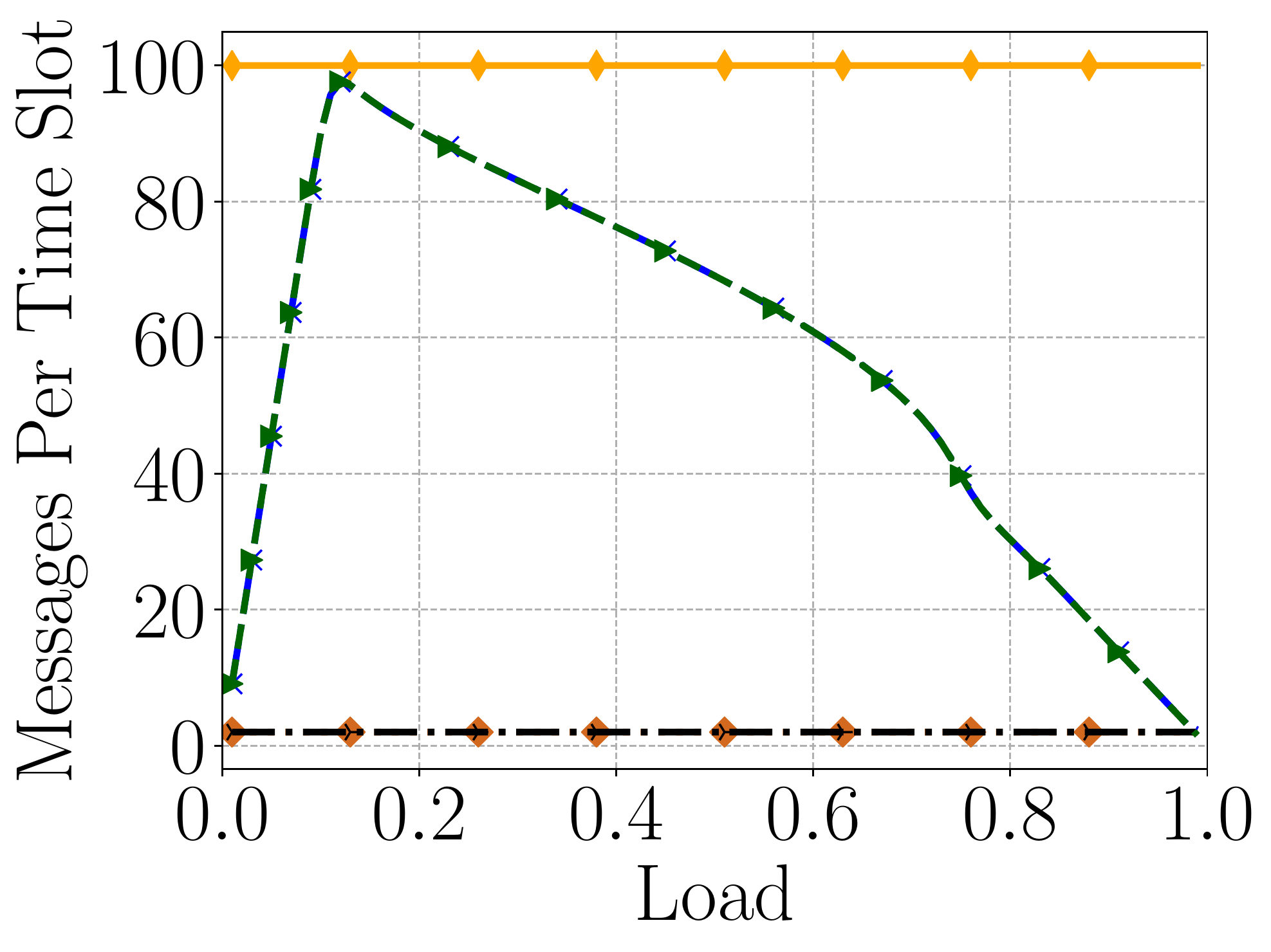}
\end{subfigure}%
\begin{subfigure}{}
\includegraphics[width=0.32\linewidth]{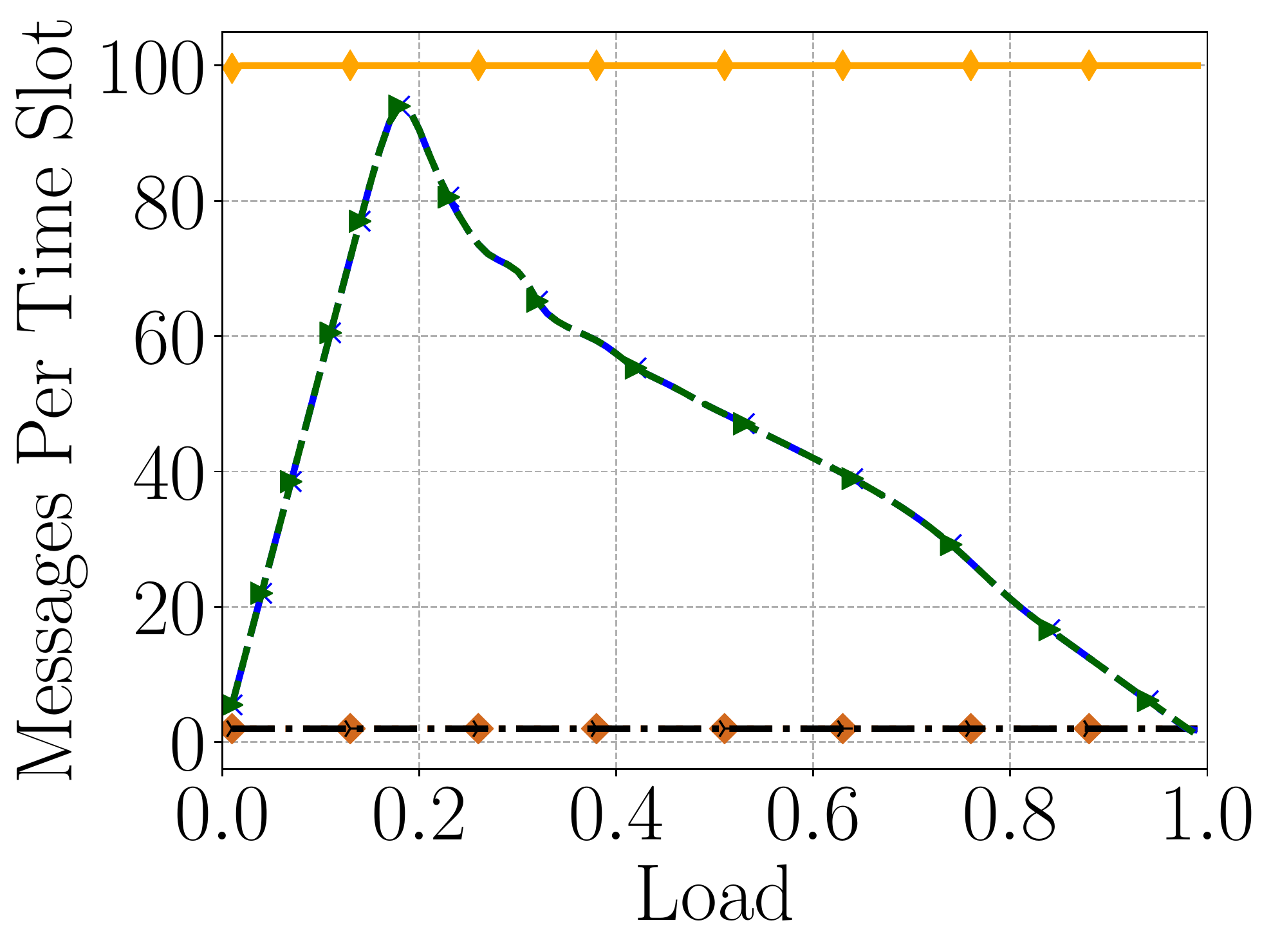}
\end{subfigure}%
\begin{subfigure}{}
\includegraphics[width=0.32\linewidth]{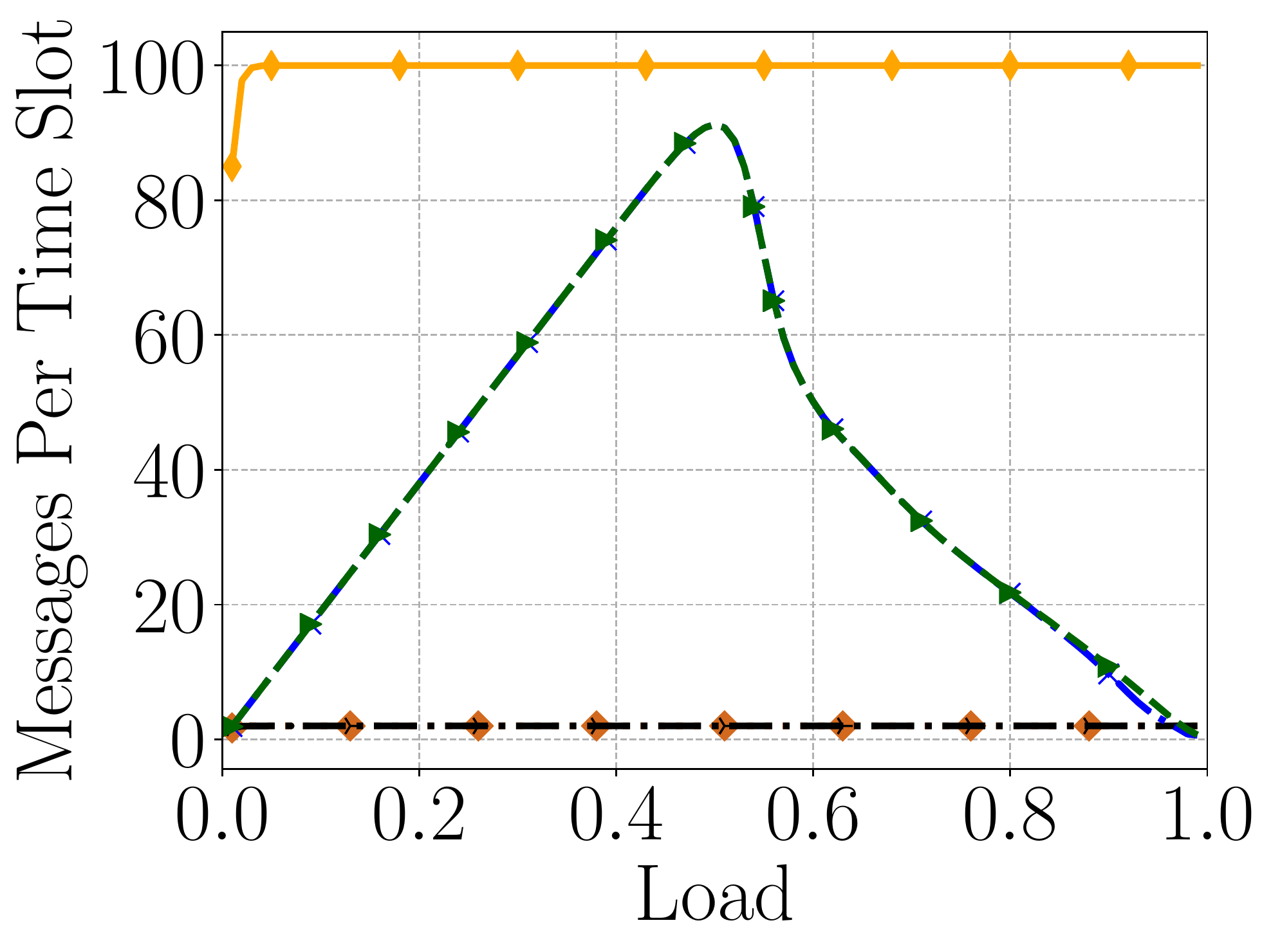}
\end{subfigure}%
\end{subfigure}
%%%%%%%%%%%%%%%%%%%%%%%%%%%%%%%%%%%%%%%%%%%%%%%%%%
\begin{subfigure}{}
\includegraphics[width=0.9\linewidth]{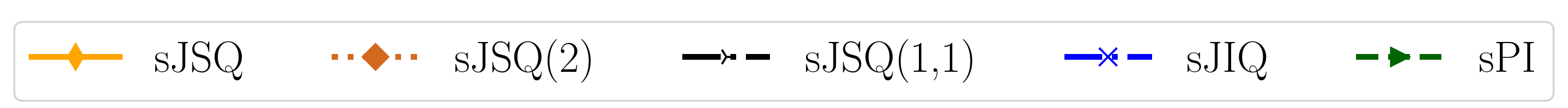}
\end{subfigure}%
%%%%%%%%%%%%%%%%%%%%%%%%%%%%%%%%%%%%%%%%%%%%%%%%%%
\caption{Performance of different policies with 100 heterogeneous servers and different slow:fast server ratios. The policies we consider here can split arriving batches of jobs. We emphasize this by adding 's' before each policy name (which stands for 'splittable').}
\label{fig:eval:splittable}
\end{figure*}

The main purpose of this section is to complement the simulation results presented in \cite{PI1} for a model with batch arrivals and random service capacities. 
To this end, we compare the performance of PI and PI-Split with that of the token-based JIQ as well as several other load-distribution policies, namely JSQ, JSQ($2$) and JSQ($1$,$1$). 

We consider systems with 100 heterogeneous servers
and study two cases corresponding to whether arriving batches can be split or not.
When batches cannot be split, the aforementioned policies route the entire batch to a single queue. When batches can be split, it is not straight-forward how the policies should split them. There are many sensible ways in which this can be done. In our simulations we chose the following splitting schemes, based on the ``water-filling'' approach (see, for example, \cite{ying2017power}):

\begin{itemize}

    \item \textbf{JSQ:} Uses the water-filling approach where the jobs belonging to the batch are sent to the servers with the smallest queue lengths in an attempt to equalize them as much as possible. For example, if there are 3 servers with queue lengths $(1,2,4)$ and a batch arrives with 3 jobs, server 1 will get 2 jobs and server 2 will get 1 job, resulting in the queue lengths $(3,3,4)$. Ties are broken randomly.
    
    \item \textbf{JSQ(2):} Uses the same water-filling approach as JSQ, but only for the two servers that were chosen uniformly at random.
    
    \item \textbf{JSQ(1,1):} Same as JSQ(2), but for the ``remembered'' server and the one chosen uniformly at random.
    
    \item \textbf{JIQ:} Uses water-filling only for the idle servers. Specifically, a batch is split as evenly as possible between the idle servers. If there are no idle servers, each job in the batch is sent to a randomly chosen server.
    
    \item \textbf{PI:} If there are idle servers, the splitting is the same as in JIQ. Otherwise, the entire batch is sent to the last idle server.
    
\end{itemize}

As shall be seen throughout this section, simulation results indicate that PI is indeed stable and offers appealing properties in comparison to JIQ and other reduced-state policies. %\GM{summarize here}.

\subsection{Simulation settings} 
We consider a system with $100$ heterogeneous servers, some being slow servers, working at a rate $\mu^{-1} = 1$, and the others being fast servers, working at a rate $(10\mu)^{-1} = 1/10$. The service time distribution of all servers is uniform and lower bounded by 1. Note that these service time distributions concur with Assumption \eqref{eq:s assumption}).
We implement three scenarios, which correspond to different partitions of the servers into the two classes. More specifically, we consider the following slow:fast ratios: $10{:}90$, $50{:}50$ and $90{:}10$.  
The number of arrivals at each time slot is drawn from a Poisson process of rate $\lambda$. We are interested in values of $\lambda$ such that $0 < \lambda < \sum_{i=1}^n \mu_i$, \ie any strictly positive and sub-critical arrival rates.

For each scenario, we run simulations for $10^6$ time slots for different loads, where all policies receive the same input process, but possibly make different routing decisions. The number of time slots was chosen such that the difference in the outputs of different runs at the maximal load was negligible. We present the following graphs:

\T{Steady state average queue length vs. load.} For each simulation run with a given load, we calculate the time averaged sum of queue lengths in the system. 
Under all policies, whenever the chain is positive recurrent, it is also ergodic. Thus the average queue length converges, and one simulation run is enough to calculate the steady-state average queue length. Also, there are points that do not appear on the graphs since they are vertically truncated. However, since the truncation is at very high average queue lengths, these missing points indicate either poor performance or instability.  

\T{CDF of JCT in steady state for 50$\%$ and 99$\%$ load.} We keep track of the job completion times of all jobs that entered the system during the simulations, with $\lambda{=}0.5\sum_{i=1}^n\mu_i$ and $\lambda{=}0.99\sum_{i=1}^n\mu_i$. We calculate the cumulative distribution function (CDF). Policies that are unstable under a given load are not shown in the graph. 

\T{Messages per time slot vs. load.} For JIQ and PI, for each simulation run with a given load, we calculate the number of times a token is passed from a server to the dispatcher (the reverse path is not counted because the token is sent along with a job). We then divide this by the number of time slots to obtain the average number of messages per time slot. Counting messages for the other policies is more straightforward. At each time slot with non-zero arrivals, for JSQ it is the number of servers needed to be probed, \ie $100$. Similarly, for JSQ(2) and JSQ(1,1) it is $2$. 

\begin{remark}
Instead of calculating messages per time slot, one can choose to look at messages per batch, per sub batches (in the splittable case) or per job. However, measuring the number of messages per time slot is more informative for the amount of needed communication, whereas the other options are affected by the arrival and departure rates rather than by their ratio (\ie normalized load).
\end{remark}

\begin{comment}
\begin{remark} The number of messages needed for the above policies highly depends on the application and the implementation. For example, if a server could inform the dispatcher whenever it finishes a job, JSQ can be implemented with just one message per job. Thus we only present a comparison for the case where JLW / JSQ / SQ(2) / SQ(1,1) require probing the servers.
\end{remark}
\end{comment}

\subsection{The unsplittable case: simulation results}

\T{Stability and average queue lengths.} The results are depicted in Figure \ref{fig:eval:unsplittable}. As can be seen, for all the different slow:fast ratios, PI appears to achieve the stability region. Also, the average queue length under PI is favorable to that of other reduced-state policies and even comparable to that of JSQ for low and moderate loads.
JSQ(1,1) also appears to achieve the stability region, but with a much larger average queue length. JSQ(2) appears to be unstable for moderately-high loads for the $50{:}50$ and $90{:}10$ ratios. 
Also, whenever JSQ(2) is stable, its average queue length is much higher than under PI. JIQ appears unstable for loads close to $100\%$ for all ratios. 

\T{Job completion time.} 
We measure the completion time of all jobs during the simulation at a moderate load of 50$\%$ and a high load of 99$\%$. It is evident how both PI and JIQ are competitive to JSQ at 50\% for all slow:fast ratios. This is because at such load there are many idle servers and both of these token-based policies resemble JSQ. JSQ(2) and JSQ(1,1) show lesser performance, and at a 50:50 ratio JSQ(2) is missing as its completion times diverge due to apparent instability.
At a load of 99$\%$, PI shows the best overall performance after JSQ. Both JIQ and JSQ(2) do not appear in 2 out of 3 scenarios due to what appears to be instability at this load. JSQ(1,1) again shows lesser performance.

\T{Messages per time slot.} Since the probability of zero arrivals in a time slot is negligible even at low loads (\ie $e^{-\lambda}$), the message overhead of JSQ, JSQ(2) and JSQ(1,1) is approximately constant. It is 100 for JSQ and 2 for JSQ(2) and JSQ(1,1). 
On the other hand, the number of messages per time slot for PI and JIQ is 1 up to a high load. As the load increases, there is an increasing number of time slots without any servers becoming idle and the message overhead for both policies decreases.

%%%%%%%%%%%%%%%%%%%%%%%%%%%%%%%%%%%%%%%%%%%%%%%%%%%%

\subsection{Splittable case: simulation results}
 
\T{Stability and average queue length.} The results are depicted in Figure \ref{fig:eval:splittable}. As can be seen, the results show a similar trend to those of the unsplittable case. It is also evident how the advantage of JSQ in comparison to the other policies increases.   

\T{Job completion time and Messages per time slot.} In the splittable case, the communication overhead of token based policies has different characteristics than in the unsplittable case. This is because in the splittable case, the overhead may reach 1 message per job rather than 1 message per batch of jobs. Indeed, this would be the case for moderate loads at which we have both a sufficient number of jobs in each batch and they are split among several/many idle servers. 

As can be seen from the graphs, at such moderate loads, both PI ans JIQ are competitive with JSQ but with often lower message overhead. JSQ(2) and JSQ(1,1) perform significantly worse but also have a lower message overhead in this case.
At the high load of 99$\%$ the message overhead of the token based policies rapidly decreases due to the decreased number of idled servers. In this case the results are similar to those of the unspittable case with apparent instability of JSQ(2) in 2 out of 3 cases and JIQ in 1 out of 3. JSQ(1,1) appears to be stable but again performs poorly.  
%%%%%%%%%%%%%%%%%%%%%%%%%%%%%%%%%%%%%%%%%%%%%%%%%%%%

\begin{remark}
We have found in our simulations that the stability region of JIQ increases when batches are allowed to be split. To the best of our knowledge, this phenomenon was not pointed out previously. We remark it as a future research challenge to provide formal analysis for this phenomenon and explore how it can be leveraged to improve the performance of token based policies.
\end{remark}

%\appendix
\section{Appendix}
\noindent \textbf{Proof of Equation \eqref{eq: RRW bound}}. To reiterate \eqref{eq: RRW bound}, we prove that 
\begin{equation}\label{eq: appendix RRW bound}
    \mathbb{E}\Big[ \Big({Q_{l_k}(\tau_{k+1})}\Big)^{1+\gamma}\mid {\cal{F}}^{\xi}_{\tau_k}\Big]\leq ([\lambda-\mu_{l_k}]^+)^{1+\gamma}\mathbb{E}[\Delta_k^{1+\gamma}\mid {\cal{F}}^{\xi}_{\tau_k}]+
    C\mathbb{E}[\Delta_k^{0.75(1+\gamma)}\mid {\cal{F}}^{\xi}_{\tau_k}],
\end{equation}
where $l_k \in [n]$ is the Last-Idle server during the interval $[\tau_k+1,\tau_{k+1}]$, $\Delta_k=\tau_{k+1}-\tau_k$ is the duration of the interval and $Q_{l_k}(\tau_{k+1})$ is the queue length at server $l_k$ at time $\tau_{k+1}$. Recall that the RVs $l_k$, $\tau_k$ and $\{Q_i(\tau_k)\}_{i\in[n]}$ are measurable w.r.t ${\cal{F}}^{\xi}_{\tau_k}$, and that by the definition of the sampling times, we must have $Q(\tau_k)=0$ and that server $l_k$ receives all incoming jobs during $[\tau_k+1,\tau_{k+1}]$.

The left hand side of \eqref{eq: appendix RRW bound} is therefore the $1+\gamma$ moment of the position, after $\Delta_k$ time slots, of a reflected random walk starting at zero. Recall the definition in \eqref{eq: beta and sigma} of $\beta_{i}:=\lambda-\mu_{i}$ and $\sigma_{i}^2:=\sigma_a^2+\sigma_{s_{i}}^2$. Given the value of $l_k$, the mean and variance of the step size of this random walk are given by $\beta_{l_k}$ and $\sigma_{l_k}^2$ respectively. 

Lemma 6 of \cite{PI1} provides an upper bound on the $1+\gamma$ moment of a very similar reflected random walk. There are two differences between that result and what we want to prove here. First, the random walk in \cite{PI1} is of the workload process, not the queue length process as we consider here. Specifically, the distribution of the step size has a different dependence on the model parameters. However, the proof in \cite{PI1} only uses the mean and variance of the step size. Therefore, the same derivation is applicable here and we just plug in our mean and variance. 

Second, the bound in \cite{PI1} is derived for \textit{a fixed number of steps}. In our case, given ${\cal{F}}^{\xi}_{\tau_k}$, the number of steps is \textit{random}. This is due to the random service capacities. But, the number of steps until the next sampling time is independent of the random walk at server $l_k$. It only depends on the other servers and how many time slots it takes one of their queue lengths to reach zero. 

Thus, a direct application of Lemma 6 in \cite{PI1} (after plugging in $\beta_{l_k}$ and $\sigma_{l_k}$) yields
\begin{equation*}
    \mathbb{E}\Big[ \Big({Q_{l_k}(\tau_{k+1})}\Big)^{1+\gamma}\mid {\cal{F}}^{\xi}_{\tau_k}, \Delta_k=t\Big]\leq ([\beta_{l_k}]^+)^{1+\gamma}t^{1+\gamma}+
    Ct^{0.75(1+\gamma)},
\end{equation*}
where 
$$C=\max_i\{\big(16\sigma_i+8(\sigma_i)
^{1/2}(\beta_i)^+\big)^{(1+\gamma)/2}\}.$$
Thus, we have
\begin{align*}
    \mathbb{E}\Big[ \Big({Q_{l_k}(\tau_{k+1})}\Big)^{1+\gamma}\mid {\cal{F}}^{\xi}_{\tau_k}\Big]&=
    \sum_{t=1}^\infty \mathbb{E}\Big[ \Big({Q_{l_k}(\tau_{k+1})}\Big)^{1+\gamma}\mid {\cal{F}}^{\xi}_{\tau_k}, \Delta_k=t\Big]\mathbbm{P}(\Delta_k=t \mid  {\cal{F}}^{\xi}_{\tau_k} )\cr
    &\leq     \sum_{t=1}^\infty \Big( ([\beta_{l_k}]^+)^{1+\gamma}t^{1+\gamma}+
    Ct^{0.75(1+\gamma)}\Big)\mathbbm{P}(\Delta_k=t \mid  {\cal{F}}^{\xi}_{\tau_k} )
    \cr
    &=([\lambda-\mu_{l_k}]^+)^{1+\gamma}\mathbb{E}[\Delta_k^{1+\gamma}\mid {\cal{F}}^{\xi}_{\tau_k}]+
    C\mathbb{E}[\Delta_k^{0.75(1+\gamma)}\mid {\cal{F}}^{\xi}_{\tau_k}].
\end{align*}

\qed
%%%%%%%%%%%%%%%%%%%%%%%%%%%%%%%%%%%%%%%%%%%%%%%%%%%%%%%%%%%%%%%%%%%%%%%%%%%%%%

\bigskip

%\balance

\bibliography{mybib}

\begin{thebibliography}{1}
\expandafter\ifx\csname url\endcsname\relax
  \def\url#1{\texttt{#1}}\fi
\expandafter\ifx\csname urlprefix\endcsname\relax\def\urlprefix{URL }\fi
\expandafter\ifx\csname href\endcsname\relax
  \def\href#1#2{#2} \def\path#1{#1}\fi

\bibitem{PI1}
R.~Atar, I.~Keslassy, G.~Mendelson, A.~Orda, S.~Vargaftik, Persistent-idle
  load-distribution, Stochastic Systems 10~(2) (2020) 152--169.

\bibitem{ying2017power}
L.~Ying, R.~Srikant, X.~Kang, The power of slightly more than one sample in
  randomized load balancing, Mathematics of Operations Research 42~(3) (2017)
  692--722.

\bibitem{lu2011join}
Y.~Lu, Q.~Xie, G.~Kliot, A.~Geller, J.~R. Larus, A.~Greenberg, Join-idle-queue:
  A novel load balancing algorithm for dynamically scalable web services,
  Performance Evaluation 68~(11) (2011) 1056--1071.

\bibitem{mitzenmacher2001power}
M.~Mitzenmacher, The power of two choices in randomized load balancing, IEEE
  Transactions on Parallel and Distributed Systems 12~(10) (2001) 1094--1104.

\bibitem{shah2002use}
D.~Shah, B.~Prabhakar, The use of memory in randomized load balancing, in:
  Proceedings IEEE International Symposium on Information Theory, 2002, p. 125.

\bibitem{asmussen2008applied}
S.~Asmussen, Applied probability and queues, Vol.~51, Springer Science \&
  Business Media, 2008.

\bibitem{malyshev1979ergodicity}
V.~A. Malyshev, M.~V. Men'shikov, Ergodicity, continuity and analyticity of
  countable markov chains, Trudy Moskovskogo Matematicheskogo Obshchestva 39
  (1979) 3--48.

\bibitem{meyn1994state}
S.~P. Meyn, R.~Tweedie, State-dependent criteria for convergence of markov
  chains, The Annals of Applied Probability (1994) 149--168.

\bibitem{durrett2019probability}
R.~Durrett, Probability: theory and examples, Vol.~49, Cambridge university
  press, 2019.

\end{thebibliography}

\end{document}